\documentclass[11pt,a4paper]{article}
\usepackage[utf8]{inputenc}
\usepackage{amsmath}
\usepackage{amsthm}
\usepackage{amsfonts}
\usepackage{algorithm}
\usepackage{algpseudocode}
\newcommand{\hlabel}{\phantomsection\label}

\usepackage{multirow}
\usepackage[table,xcdraw]{xcolor}
\usepackage{subcaption}
\usepackage{float}
\usepackage{rotating}
\usepackage{hyperref}
\usepackage{tabularx}
\usepackage{multicol}
\usepackage{tikz}
\usepackage{amssymb}
\usepackage{blindtext}
\usepackage{graphicx}
\usepackage[shortlabels]{enumitem}
\usepackage[left=2cm,right=2cm,top=2cm,bottom=2cm]{geometry}
\usepackage{booktabs}
\usepackage[table,xcdraw]{xcolor}
\usepackage{multirow}
\usepackage{rotating,tabularx}
\usepackage{longfbox}

\newtheorem{Step}{Step}
\newtheorem{Theorem}{Theorem}[section]
\newtheorem{Proposition}[Theorem]{Proposition}
\newtheorem{Remark}[Theorem]{Remark}
\newtheorem{Lemma}[Theorem]{Lemma}

\newtheorem{Definition}[Theorem]{Definition}
\newtheorem{Example}{Example}
\newtheorem{Algorithm}{Algorithm}
\usepackage{hyperref}
\hypersetup{colorlinks=true,urlcolor=red}
\expandafter\let\expandafter\oldproof\csname\string\proof\endcsname
\let\oldendproof\endproof
\renewenvironment{proof}[1][\proofname]{
\oldproof[\ttfamily\scshape \bf #1.]
}{\oldendproof}
\newcommand{\set}[1]{\left\{#1\right\}}

\def\emp{\emptyset}

\def\dom{{\rm dom}\,}

\def\ox{\overline{x}}

\def\disp{\displaystyle}

\def\Bar{\overline}
\def\ra{\rangle}
\def\la{\langle}

\def\epsilon{\varepsilon}
\def\ox{\bar{x}}

\def\prox{\mbox{\rm Prox}}

\def\dom{\mbox{\rm dom}\,}

\def\emp{\emptyset}
\def\st{\stackrel}
\def\oR{\Bar{\R}}

\def \N{{\rm I\!N}}
\def \R{{\rm I\!R}}

\def\vt{\vartheta}

\newcommand{\dotproduct}[1]{\left\langle#1\right\rangle}
\newcommand{\brac}[1]{\left(#1\right)}

\newcommand{\norm}[1]{\left\|#1\right\|}
\numberwithin{equation}{section}

\title{\bf A New Inexact Gradient Descent Method \\
with Applications to Nonsmooth Convex Optimization}
\author{Pham Duy Khanh\footnote{Department of Mathematics, Ho Chi Minh City University of Education, Ho Chi Minh City, Vietnam. E-mails: pdkhanh182@gmail.com, khanhpd@hcmue.edu.vn. Research of this 
 author is funded by the Ministry of Education and Training Research Funding under the project ``\textit{First-order approximation methods and applications}''}\quad Boris S. Mordukhovich\footnote{Department of Mathematics, Wayne State University, Detroit, Michigan, USA. E-mail: aa1086@wayne.edu. Research of this author was partly supported by the US National Science Foundation under grants DMS-1808978 and DMS-2204519, by the Australian Research Council under grant DP-190100555, and by Project 111 of China under grant D21024.}\quad Dat Ba Tran\footnote{Department of Mathematics, Wayne State University, Detroit, Michigan, USA. E-mail: tranbadat@wayne.edu. Research of this author was partly supported by the US National Science Foundation under grants DMS-1808978 and DMS-2204519.}}
\begin{document}
\maketitle
\vspace*{-0.3in}
\begin{center}
 {\bf Dedicated to the memory of Andreas Griewank,\\ an outstanding mathematician and a wonderful human being} 
\end{center}

\noindent
{\small{\bf Abstract}. The paper proposes and develops a novel inexact gradient method (IGD) for minimizing ${\cal C}^1$-smooth functions with Lipschitzian gradients, i.e., for problems of ${\cal C}^{1,1}$ optimization. We show that the sequence of gradients generated by IGD converges to zero. The convergence of iterates to stationary points is guaranteed under the Kurdyka-\L ojasiewicz (KL) property of the objective function with convergence rates depending on the KL exponent. The newly developed IGD is applied to designing two novel gradient-based methods of nonsmooth convex optimization such as the inexact proximal point methods (GIPPM) and the inexact augmented Lagrangian method (GIALM) for convex programs with linear equality constraints. These two methods inherit global convergence properties from IGD and are confirmed by numerical experiments to have practical advantages over some well-known algorithms of nonsmooth convex optimization.}\\[1ex] 

 \section{Introduction}\label{sec:intro}

 The paper addresses numerical optimization, the area to which the contribution of Andreas Griewank is difficult to overstate. In particular, his book with Andrea Walther \cite{andreas} is the classical collection of constructive methods of algorithmic differentiation in nonlinear optimization being a strong inspiration for a great many of applied mathematicians and practitioners.\vspace*{0.03in}

This paper is devoted to developing new algorithms to solve various classes of optimization problems. Let us start with the following basic problem of {\em ${\cal C}^{1,1}$ optimization}:
\begin{align}\label{optim prob}
{\rm minimize}\quad f(x)\quad\text{ subject to }\; x\in\R^n,
\end{align}
where $f:\R^n\rightarrow\R$ is a continuously differentiable ($\mathcal{C}^1$-smooth) objective function whose gradient $\nabla f$ is globally Lipschitzian on $\R^n$ with some constant $L>0$. A very natural and classical approach to solve optimization problems of type (\ref{optim prob}) is by using the \textit{gradient descent method} described as follows: given a starting point $x^1\in\R^n$, construct the iterative procedure 
\begin{align}\label{intro: stepsize + descent direction}
x^{k+1}:=x^k-\dfrac{1}{L}\nabla f(x^k)\;\text{ for all }\;k\in\N.
\end{align}
Largely due to its simplicity, the gradient descent method is broadly used to solve various optimization problems; see, e.g., \cite{beckbook,bottou18,curry44,nesterovbook18,nielsen15}. However, errors in gradient calculations may appear deterministically in various contexts. In derivative-free optimization problems where only the function values are accessible, the exact gradients are not available but only their approximations via finite differences and simplex gradients are \cite{masoud,audet17,conn09}. In addition, many optimization methods for minimizing nonsmooth functions {\cite{devolder14,nel14,pas17,rockafellar76,rockafellar alm,themelis20,themelis17} can be treated via the gradient descent method applied to their regularized functions. Errors in gradient calculations for regularized functions, which appear in practical situations, require developing gradient descent methods with {\em inexact gradient information}.

For these reasons, gradient methods taking errors into account have been developed over the years. Among the major ones, we list the following. Gilmore and Kelley \cite{gilmore95} proposed the implicit filtering algorithm for minimizing a noisy smooth function under box constraints, which uses finite difference approximations of the exact gradient. Bertsekas and Tsitsiklis \cite{bertsekas00} justified the convergence of gradients in gradient methods with errors smaller than the multiplication of the stepsize by an expression depending on the exact gradient. The convergence for incremental and stochastic gradient methods was also established in \cite{bertsekas00}. Nesterov proposed in \cite{nesterov05} and further developed in \cite{nesterov15} and his paper with Devolder and  Glineur
\cite{devolder14} gradient methods with inexact oracle for convex functions. {Generalizations of the inexact oracle are discussed in \cite{bog16, dvu17} for nonconvex functions. Additionally, accelerated gradient methods considering various types of errors are analyzed in \cite{vasin23}. Optimization methods with stochastic errors are also studied in \cite{ghadimi13,ghadimi13}.} Quite recently, the authors of the present paper \cite{kmt22.1} proposed a general framework of inexact reduced gradient (IRG) methods for minimizing nonconvex $\mathcal{C}^1$-smooth function with or without the Lipschitz continuity of the gradient under the condition that the distance between the exact and the approximate gradient is less than a determined error at each iteration. To the best of our knowledge, \cite{kmt22.1} is the first paper addressing the global convergence and convergence rates of inexact gradient descent methods with deterministic error for general classes of nonconvex functions.\vspace*{0.03in}

The present paper designs and justifies the novel {\em inexact gradient descent} (IGD) method to solve problem \eqref{optim prob} by using some IRG ideas from \cite{kmt22.1}. The iterative procedure of IGD is \eqref{intro: stepsize + descent direction} with the exact gradient replaced by one of its approximations. The main differences between IGD and its predecessor IRG are  that the construction of IGD in Algorithm~\ref{al gd inexact} use directly an approximate gradient as a descent direction instead of introducing another reduced gradient direction as in \cite[Algorithm~1]{kmt22.1}. The convergence analysis of the new IGD method does not involve the notion of null iterations which is essential in the results of \cite[Propositions~4.6, 4.7]{kmt22.1} and the construction of non-null iteration as in \cite[(5.23)]{kmt22.1}. The IGD method also achieves the following fundamental convergence properties:

\begin{itemize}
\item Convergence to zero of the gradient sequence.

\item Convergence to some stationary point of the iterative sequence under the Kurdyka-\L ojasiewicz (KL) property of the objective function.

\item Convergence rates depending on the KL exponent of the objective function for the sequences of iterates, gradients, and function values.
\end{itemize}

It should be mentioned that the results in \cite{devolder14,nesterov05,nesterov15} address only convex optimization problems. Although the implicit filtering method in \cite{gilmore95} deals with nonconvex problems, it does not achieve any of the fundamental convergence properties above due to the presence of noise. The gradient methods with errors in \cite{bertsekas00} establish only the convergence to zero of the gradient sequence. Moreover, the rates of convergence for the gradient sequence and function value sequence are not provided in \cite{kmt22.1}.\vspace*{0.05in}

The IGD method developed in this paper for ${\cal C}^{1,1}$ optimization problems of type \eqref{optim prob} is applied then to problems of {\em nonsmooth optimization} to design and justify the following gradient-based methods:

\begin{itemize}
\item The {\em gradient-based inexact proximal point method} (GIPPM) for minimizing general nonsmooth convex functions.

\item The {\em gradient-based inexact augmented Lagrangian method} (GIALM) for minimizing nonsmooth convex functions under linear equality constraints.
\end{itemize}

The two methods listed above inherit the convergence properties from IGD with establishing in this way the stationarity of accumulation points and the global convergence under the KL property of the objective function with convergence rates depending on the KL exponent. Regarding application aspects, the conducted numerical experiments show that the GIALM method with the new handling of errors exhibit a better performance than the classical inexact augmented Lagrangian  method by Rockafellar \cite{rockafellar alm} in practical problems of image processing as well as in random generated examples.\vspace*{0.05in}

 The rest of the paper is organized as follows. Section~\ref{sec prelim} presents some preliminaries and discussions. The construction and convergence analysis of the main IGD method are given in Section~\ref{sec Stationarity of accumulation points}. The next Section~\ref{sec:5}, which is split into two subsections, proposes and justifies the new GIPPM and GIALM methods of nonsmooth convex optimization with their convergence analysis. The numerical experiments and comparisons between the newly developed  inexact gradient-based methods with the existing algorithms are given in the final Section~\ref{sec:6}. 

\section{Preliminaries and Discussions}\label{sec prelim}

First we recall some notions and notation frequently used in the paper. All our considerations are given in the space $\R^n$ with the Euclidean norm $\|\cdot\|$. We use the matrix norm defined by
\begin{equation*}
\norm{A}:=\max\big\{\norm{Ax}\;\big|\;\norm{x}=1\big\}\;\text{ for any }\;m\times n \text{ matrix }A.
\end{equation*}
Along with the Euclidean norm $\norm{\cdot}$, in Section~\ref{sec:6} we also use the $\ell_1$ norm $\norm{\cdot}_1$ denoted by
\begin{equation*}
\disp\norm{x}_1=\sum_{k=1}^n|x_k|\;\text{ for all }\;x=(x_1,\ldots,x_n)\in\R^n.
\end{equation*}
As always, $\N:=\{1,2,\ldots\}$ signifies the collections of natural numbers. The symbol $x^k\st{J}{\to}\ox$ means that $x^k\to\ox$ as $k\to\infty$ with $k\in J\subset\N$. Recall that $\bar x$ is a \textit{stationary point} of a $\mathcal{C}^1$-smooth function $f\colon\R^n\rightarrow\R$ if $\nabla f(\bar x)=0$. A $\mathcal{C}^1$-smooth function $f:\R^n\rightarrow \R$ is said to have a Lipschitz continuous gradient with the uniform constant $L>0$ on $\R^n$ if 
\begin{align*}
\norm{\nabla f(x)-\nabla f(y)}\le L\norm{x-y}\;\text{ for all }\;x,y\in\R^n.
\end{align*}
It follows from \cite[Lemma~A.11]{solodovbook} that the Lipschitz continuity of $\nabla f$ with constant $L>0$ implies that 
 \begin{align}\label{descent condition}
f(y)\le f(x)+\dotproduct{\nabla f(x),y-x}+\dfrac{L}{2}\norm{y-x}^2\;\text{ for all }\;x,y\in\R^n.
\end{align}
The converse implication may not hold in general as discussed in \cite[Section~2]{kmt22.1}.\vspace*{0.05in}
 
The following  {\em Kurdyka-\L ojasiewicz property} is taken from Absil et al. \cite[Theorem~3.4]{absil05}.

\begin{Definition}\rm \label{KL ine}\rm
Let $f:\R^n\rightarrow\R$ be a differentiable function. We say that $f$ satisfies the \textit{KL property} at $\bar{x}\in\R^n$ if there exist a number $\eta>0$, a neighborhood $U$ of $\bar{x}$, and a nondecreasing function $\psi:(0,\eta)\rightarrow(0,\infty)$ such that the function $1/\psi$ is integrable over $(0,\eta)$ and we have
\begin{align}\label{kl absil}
 \norm{\nabla f(x)}\ge\psi\big(f(x)-f(\bar{x})\big)\;\mbox{ for all }\;x\in U\;\mbox{ with }\;f(\bar x)<f(x)<f(\bar x)+\eta.
\end{align}
\end{Definition}

\begin{Remark}\rm\label{algebraic}
It has been realized that the KL property  is satisfied in broad settings. In particular, it holds at every {\em nonstationary point} of $f$; see \cite[Lemma~2.1~and~Remark~3.2(b)]{attouch10}.  Furthermore, it is proved at the seminal paper by \L ojasiewicz \cite{lojasiewicz65} that any analytic function $f:\R^n\rightarrow\R$ satisfies the KL property at every point $\ox$ with $\psi(t)~=~Mt^{q}$ for some $q\in [0,1)$. As demonstrated in \cite[Section~2]{kmt22.1}, the KL property formulated in Attouch et al. \cite{attouch10} is stronger than the one in Definition~\ref{KL ine}. Typical smooth functions that satisfy the KL property from \cite{attouch10}, and hence the one from Definition~\ref{KL ine}, are smooth {\em semialgebraic} functions and also those from the more general class of functions known as {\em definable in o-minimal structures}; see \cite{attouch10,attouch13,kurdyka}. The following preservation properties for general semialgebraic functions can be found, e.g., in \cite{attouch10}:  
\begin{itemize}
\item Finite sums and products of semialgebraic functions are semialgebraic.
\item Functions of the marginal type type 
\begin{equation*}
\R^n \ni x \mapsto f(x)=\inf_{y\in\R^m} g(x,y),
\end{equation*} 
where $g$ is a semialgebraic function of two variables, are semialgebraic.
\end{itemize}
 \end{Remark}

 Next we present based on \cite{absil05} some descent-type conditions ensuring the global convergence of iterates for smooth functions that satisfy the KL property. 
 
\begin{Proposition}\label{general convergence under KL}
Let $f:\R^n\rightarrow\R$ be a $\mathcal{C}^1$-smooth function, and let the following conditions hold along a sequence of iterates $\set{x^k}\subset\R^n$ for the function $f$:
\begin{itemize}
\item[\bf(H1)] {\rm(primary descent condition)}. There is $\sigma>0$ such that we have 
\begin{align*}
 f(x^k)-f(x^{k+1})\ge\sigma\norm{\nabla f(x^k)}\cdot\norm{x^{k+1}-x^k}
\end{align*}
{for sufficiently large $k\in\N$.}
\item[\bf(H2)] {\rm(complementary descent condition)}. {The following implication holds:}
\begin{align*}
 \big[f(x^{k+1})=f(x^k)\big]\Longrightarrow [x^{k+1}=x^k]
\end{align*}
{for sufficiently large $k\in\N$.}
\end{itemize}
If $\bar x$ is an accumulation point of $\set{x^k}$  and $f$ satisfies the KL property at $\bar x$, then $x^k\rightarrow\bar x$ as $k\to\infty$.
\end{Proposition}

When the sequence under consideration is generated by a linesearch method and satisfies some stronger conditions than (H1) and (H2) in 
Proposition~\ref{general convergence under KL}, its convergence rates are established in \cite{kmt22.1} under the KL property with $\psi(t)=Mt^{q}$ as given below.

\begin{Proposition}\label{general rate}
Let $f:\R^n\rightarrow\R$ be a $\mathcal{C}^1$-smooth function. {Assume that the sequences $\set{x^k}$ and $\set{d^k}$ satisfy the iterative update $x^{k+1}=x^k+\tau d^k$ with $x^{k+1}\ne x^k$ for all $k\in\N$, and that 
\begin{align}\label{two conditions}
f(x^k)-f(x^{k+1})\ge \alpha\norm{d^k}^2\;\text{ and }\;\norm{\nabla f(x^k)}\le \beta\norm{d^k}
\end{align}
for sufficiently large $k\in\N$ with some constants $\tau,\alpha,\beta>0$.} Suppose  in addition that $\bar x$ is an accumulation point of $\set{x^k}$ and that $f$ satisfies the KL property at $\bar x$ with $\psi(t)=Mt^{q}$ for some $M>0$ and $q\in(0,1)$. Then the following convergence rates are guaranteed:
\begin{itemize}\vspace*{-0.05in}
\item[\bf(i)] For $q\in(0,1/2]$, the sequence $\set{x^k}$ converges linearly to $\bar x$ as $k\to\infty$.\vspace*{-0.05in}
\item[\bf(ii)] For $q\in(1/2,1)$, we have  $\norm{x^k-\bar x}=\mathcal{O}( k^{-\frac{1-q}{2q-1}})$ as $k\to\infty$
\end{itemize}
\end{Proposition}

\section{Inexact Gradient Descent Method}\label{sec Stationarity of accumulation points}

In this section, which is split into two {subsections}, we design and provide a convergence analysis of our main {\em inexact gradient method} (IGD) to solve ${\cal C}^{1,1}$ optimization problems of type \eqref{optim prob}.

\subsection{Algorithm Formulation and Motivating Examples}

Here is the proposed IGD algorithm to solve \eqref{optim prob} by using inexact gradients.

\begin{Algorithm}[IGD]\rm \label{al gd inexact}
\quad
\setcounter{Step}{-1}
\begin{Step}[initialization]\rm Select an initial point $x^1\in \R^n$, initial error $\varepsilon_1$, error reduction factor $\theta\in(0,1)$, and error scaling factor $\mu>1.$ Set $k:=1$.
\end{Step}
\begin{Step}[inexact gradient]\rm \label{Step 1}
Find  $g^k$ and the smallest natural number $i_k$ such that
\begin{align}\label{calculate approx grad}
\norm{g^k-\nabla f(x^k)}\le \theta^{i_k}\varepsilon_k \text{ and }\norm{g^k}> \mu \theta^{i_k}\varepsilon_k.
\end{align}
\end{Step}
\begin{Step}[iteration and error update]\rm \label{Step 2}
 Set $x^{k+1}:=x^{k}-\frac{1}{L}g^k$ and $\varepsilon_{k+1}:=\theta^{i_k}\varepsilon_k.$ Increase $k$ by $1$ and go back to Step \ref{Step 1}.
\end{Step}
\end{Algorithm}

Let us discuss some important features  of Algorithm~\ref{al gd inexact} and present an illustrating figure.

\begin{Remark}\rm \label{remark IGD} We have the following observations:
 \begin{enumerate}[\bf(i)]
\item \textit{The existence of $g^k$ and $i_k$ in Step~{\rm 1}}: The procedure of finding $g^k$ and $i_k$ that satisfy Step~\ref{Step 1} can be given as follows. Set $i_k=0$ and find some $g^k$ such that 
\begin{align}\label{recal}
\norm{g^k-\nabla f(x^k)}\le \theta^{i_k}\varepsilon_k.
\end{align}
While $\norm{g^k}\le \mu\theta^{i_k}\varepsilon_k$, increase $i_k$ by $1$ and recalculate $g^k$ under \eqref{recal}. When $\nabla f(x^k)\ne 0$, the existence of $g^k$ and $i_k$ in Step~1 is guaranteed. Indeed, otherwise we get a sequence $\set{g^k_i}$ with
\begin{align*}
\norm{g^k_i-\nabla f(x^k)}\le \theta^i\varepsilon_k\;\text{ and }\;\norm{g^k_i}\le \mu \theta^i\varepsilon_k\;\text{ for all }\;i\in\N.
\end{align*}
Since $\theta\in(0,1)$, this implies that $\nabla f(x^k)=0$, a contradiction. {In fact, we can bound the number $i_k$ by considering the two cases as follows.}
    
    {\textit{Case 1:} $\varepsilon_k<\frac{1}{\mu+1}\norm{\nabla f(x^k)}$. In this case, we show that $i_k=0.$}
    
    {To proceed, find some $g^k$ such that $\norm{g^k-\nabla f(x^k)}\le \varepsilon_k$. Combining this with  $\varepsilon_k<\frac{1}{\mu+1}\norm{\nabla f(x^k)}$ and with the triangle inequality yields
    \begin{align*}
       \norm{g^k}&\ge \norm{\nabla f(x^k)}- \norm{g^k-\nabla f(x^k)}\\
      &\ge \norm{\nabla f(x^k)}-\varepsilon_k \\
      &>\norm{\nabla f(x^k)}- \frac{1}{\mu+1}\norm{\nabla f(x^k)}\\
       &=\frac{\mu}{\mu+1}\norm{\nabla f(x^k)}> \mu\varepsilon_k.
    \end{align*}
    This means that $g^k$ satisfies the desired condition with $i_k=0.$}
    
    {\textit{Case 2:} $\varepsilon_k\ge\frac{1}{\mu+1}\norm{\nabla f(x^k)}.$ In this case, we show that $i_k\le \log_\theta \brac{\frac{1}{\varepsilon_k(\mu+1)}\norm{\nabla f(x^k)}}+1$.}
       
 {To proceed, assume while arguing by contradiction that $i_k> \log_\theta \brac{\frac{1}{\varepsilon_k(\mu+1)}\norm{\nabla f(x^k)}}+1$. Define $j:=i_k-1$ and get that $j>0.$ We show that for any $g^k\in\R^n$ such that $\norm{g^k-\nabla f(x^k)}\le \theta^j\varepsilon_k$, the estimate $\norm{g^k}>\theta^j\varepsilon_k$ holds, which violates the construction of $i_k$. It follows from $j> \log_\theta \brac{\frac{1}{\varepsilon_k(\mu+1)}\norm{\nabla f(x^k)}}$ and $\theta\in (0,1)$ that
    \begin{align*}
        \theta^{j}\varepsilon_k<\frac{1}{\mu+1}\norm{\nabla f(x^k)}.
    \end{align*}
    Now take any $g^k\in\R^n$ such that $\norm{g^k-\nabla f(x^k)}\le \theta^{j}\varepsilon_k$ and deduce that  
\begin{align*}
    \norm{g^k-\nabla f(x^k)}\le \theta^{j}\varepsilon_k<\frac{1}{\mu+1} \norm{\nabla f(x^k)}.
\end{align*}
Applying the triangle inequality gives us
    \begin{align*}
       \norm{g^k}&\ge \norm{\nabla f(x^k)}- \norm{g^k-\nabla f(x^k)}\\
       &>\norm{\nabla f(x^k)}- \frac{1}{\mu+1}\norm{\nabla f(x^k)}\\
       &=\frac{\mu}{\mu+1}\norm{\nabla f(x^k)}> \mu\theta^j\varepsilon_k.
    \end{align*} 
    Therefore, $i_k\le \log_\theta \brac{\frac{1}{\varepsilon_k(\mu+1)}\norm{\nabla f(x^k)}}+1$.}

\item \textit{Geometric illustration of \eqref{calculate approx grad}}: It follows from Step~1 and Step~2 of Algorithm~\ref{al gd inexact} that 
\begin{align}\label{epsilon k+1}
\norm{g^k-\nabla f(x^k)}\le \varepsilon_{k+1}\;\text{ and }\;\norm{g^k}>\mu\varepsilon_{k+1}\;\text{ for all }\;k\in\N.
\end{align}
Since $\mu>1$, the two conditions in \eqref{epsilon k+1} make the angle between $\nabla f(x^k)$ and $g^k$ smaller than $90^0$, which ensures that $-g^k$ is a descent direction of $f$ at $x^k$. As illustrated in Figure~\ref{fig:geometric representation}, the angle between these vectors can be chosen arbitrarily small by increasing $\mu$.

\begin{figure}[H]
\centering
\includegraphics[width=0.7\textwidth]{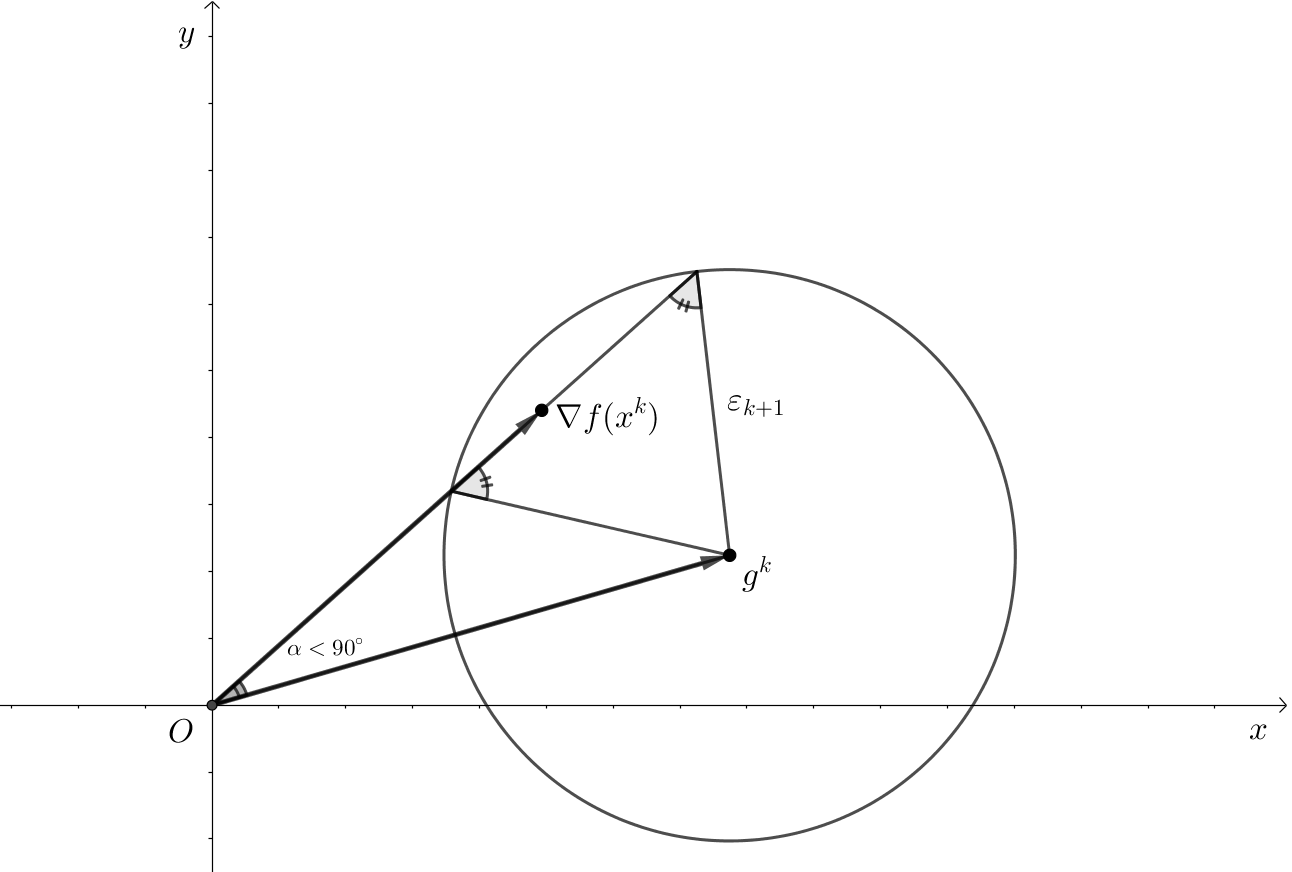}
\caption{Geometric illustration of \eqref{calculate approx grad}}
\label{fig:geometric representation}
\end{figure}
 \end{enumerate}
\end{Remark}

Note that the problem of finding $g^k$ satisfying the first condition in \eqref{calculate approx grad} can be considered more generally as follows. Given any point $x\in\R^n$ and any error value $\varepsilon>0$, find an approximation $\mathcal{G}$ of the gradient $\nabla f(x)$ such that
\begin{align}\label{universal condition}
\norm{\mathcal{G}-\nabla f(x)}\le \varepsilon.
\end{align}

Let us present some examples where the approximate condition \eqref{universal condition} appears in practical situations.

\begin{Example}[Gradient approximation methods]\rm \label{example 1} In many problems of {\em derivative-free optimization} \cite{audet17,conn09}, we only have access to function values, which can be used to derive approximate gradients satisfying \eqref{universal condition} by employing gradient approximation methods.

\begin{enumerate}
\item \textit{Forward finite difference $($FFD$)$}: For a fixed number $\delta>0$, the FFD formula gives us the approximation of the gradient $\nabla f(x)$ defined by
\begin{align}
\mathcal{G}:=\dfrac{1}{\delta}\sum_{i=1}^n\big(f(x+\delta e_i)-f(x)\big)e_i,
\end{align}
where $e_i$ is the $i^{\rm th}$ vector in the standard basis of $\R^n$. The following error bound is standard and can be found, e.g., in \cite[Section~7.2]{nocedalbook} and \cite[Theorem~2.1]{berahas22}:
\begin{align}\label{errb fordif}
\norm{\mathcal{G}-\nabla f(x)}\le \dfrac{L\sqrt{n}\delta}{2}.
\end{align}
Therefore, the choice of $\delta\in \brac{0,\frac{2\varepsilon}{L\sqrt{n}}}$ guarantees the error bound \eqref{universal condition}.

\item \textit{Centered finite difference $($CFD$)$}: For a given number $\delta>0$, the CFD formula gives us the approximation $\nabla f(x)$ defined by
\begin{align}
\mathcal{G}:=\dfrac{1}{\delta}\sum_{i=1}^n\Big(f\big(x+\frac{\delta}{2}e_i\big)-f\big(x-\dfrac{\delta}{2}e_i\big)\Big)e_i.
\end{align}
Assume that $f$ is twice continuously differentiable ($\mathcal{C}^2$-smooth) having the Lipschitz continuous Hessian with constant $M>0$. It follows from \cite[Theorem~2.2]{berahas22} that
\begin{align}\label{errb cendif}
\norm{\mathcal{G}-\nabla f(x)}\le \dfrac{\sqrt{n}M\delta^2}{24}.
\end{align}
Therefore, the choice of $\delta\in \brac{0,\sqrt{\frac{24\varepsilon}{M\sqrt{n}}}}$ guarantees the error bound required in \eqref{universal condition}.
\end{enumerate}
\end{Example}

When $f$ is smoothed from another (possibly {\em nonsmooth}) function, the exact information about {the} function value and its gradient is usually not available. However, some specific smoothing techniques allow us to find an approximation of the gradient that satisfies condition \eqref{universal condition}.

\begin{Example}[Moreau envelopes]\rm Let $g:\R^n\rightarrow\overline{\R}$ be a proper, lower semicontinuous (l.s.c.), convex function, and let $f:=e_\lambda g$ be its Moreau envelope with proximal parameter $\lambda>0$; see Section~\ref{sec:5} for more details. By \eqref{gradientMoreau}, the problem of calculating $\nabla f(x)$ approximately is equivalent to the approximate calculation of $\prox_{\lambda g}(x)$ approximately. The latter problem is equivalent to 
\begin{align}\label{minimize proximal}
\mathop{\rm minimize}_{y\in\R^n} \quad \varphi_\lambda(y):=g(y)+\dfrac{1}{2\lambda}\norm{y-x}^2.
\end{align}
Let $\bar y$ be such that $\varphi_\lambda (\bar y)-\inf_{y\in\R^n}\varphi_\lambda(y)\le \dfrac{\lambda\varepsilon^2}{2}$, and set $ \mathcal{G}:=\lambda^{-1}(x-\bar y)$. Since $\varphi_\lambda$ is strongly convex with positive constant $1/\lambda$, we deduce from \eqref{gradientMoreau} that
\begin{align*}
\norm{\mathcal{G}-\nabla f(x)}=\lambda^{-1}\norm{\bar y-\prox_{\lambda g}(x)}&\le\lambda^{-1}\sqrt{2\lambda \brac{\varphi_\lambda(\bar y)-\inf_{y\in\R^n}\varphi_\lambda (y)}}\\
&\le \lambda^{-1}\cdot \sqrt{2\lambda \dfrac{\lambda}{2}\varepsilon^2}=\varepsilon.
\end{align*}
Therefore, $\mathcal{G}$ is an approximation of $\nabla f(x)$ with the error bound $\varepsilon$.
\end{Example}

\subsection{Convergence Analysis}\label{sec convergence ana}

In this subsection, we establish the convergence properties of Algorithm~\ref{al gd inexact} including the convergence of the sequence of gradients (to zero), the sequence of iterates, and the sequence of function values together with their convergence  rates under the KL property of the objective function. Let us begin by connecting the conditions in \eqref{epsilon k+1} with the following two conditions on the gradient estimates:
\begin{align}\label{norm condition 1}
 &\norm{g^k-\nabla f(x^k)}\le \nu_1 \norm{\nabla f(x^k)}\;\text{ for all }\;k\in\N,\\
&\norm{g^k-\nabla f(x^k)}\le \nu_2 \norm{g^k}\;\text{ for all }\;k\in\N,\label{norm condition 2}
\end{align}
with some $\nu_1,\nu_2\in[0,1).$ Condition \eqref{norm condition 1} was first introduced and studied by Polyak \cite[Section~4.2.3]{polyakbook} for the convergence of inexact gradient methods for strongly convex functions. In \cite{berahas18}, the convergence of a general linesearch algorithm with noise is established when \eqref{norm condition 1} is satisfied with high probability. However, since $\nabla f(x^k)$ is unknown, it is not easy to ensure \eqref{norm condition 1}. In fact, the authors in \cite{berahas22} write: 

\medskip 

``{\em Clearly, unless we know $\norm{\nabla f(x^k)}$, condition \eqref{norm condition 1} may be hard or impossible to verify or guarantee."}

\medskip

Regarding condition \eqref{norm condition 2}, it was studied in \cite{carter91} for trust-region methods with inexact gradients for minimizing $\mathcal{C}^1$-smooth and  bounded from below functions with Lipschitzian gradients. For linesearch methods, condition \eqref{norm condition 2} is a standing assumption for investigating complexity of the inexact gradient method applied to $\mathcal{C}^2$-smooth \textit{strongly convex} functions in \cite{byrd12}, and also for convergence properties of the inexact gradient method applied to {\em convex smooth} functions in \cite{lan16}. 

 It turns out that the sequence $\set{g^k}$ generated by Algorithm \ref{al gd inexact} gives us a \textit{universal approach} to ensure both conditions \eqref{norm condition 1} and \eqref{norm condition 2} when $\mu$ is chosen properly. Indeed, it follows from \eqref{epsilon k+1} that \eqref{norm condition 2} is satisfied with $\nu_2=1/\mu$ if $\mu>1$. Applying the triangle inequality in \eqref{norm condition 2} brings us to \eqref{norm condition 1} with $\nu_1=\frac{1}{\mu-1}\in(0,1)$ whenever $\mu>2$. In addition to this observation, Step~\ref{Step 1} of Algorithm~\ref{al gd inexact} also tells us that the error between $g^k$ and $\nabla f(x^k)$ is chosen to be the largest one among $\set{\theta^{i}\varepsilon_k\;|\;i\in\N}$, which saves the executed time for finding $g^k$ in practice. In conclusion, Algorithm~\ref{al gd inexact} uses the gradient errors automatically controlled to be as large as possible while maintaining conditions \eqref{norm condition 1} and \eqref{norm condition 2}.\vspace*{0.05in}
 
Although conditions \eqref{norm condition 1} and \eqref{norm condition 2} are used in many contexts for linesearch methods, the global convergence of  the sequence of gradients (to zero), the sequence of iterates, and the sequence of function values together with their convergence rates \textit{were not established} in the studies listed above in the case of \textit{nonconvex functions}. These properties are now derived in the next theorem.

\begin{Theorem}\label{theo norm con}
Let $\set{x^k}$ and $\set{g^k}$ be the sequences satisfying the iterative procedure $x^{k+1}=x^k-\frac{1}{L}g^k$ for all $k\in\N$ under the condition \eqref{norm condition 2} with $\nu_2<\frac{1}{2}$. Assume that $\nabla f(x^k)\ne 0$ for all $k\in\N$. Then either $\norm{x^k}\rightarrow\infty$, or we have the assertions:
\begin{enumerate}[\bf(i)]
\item Both sequences $\set{\nabla f(x^k)}$ and $\set{g^k}$ converge to $0\in \R^n$ as $k\to\infty$.

\item If $f$ satisfies the KL property at some accumulation point $\bar x$ of $\set{x^k}$, then $x^k\rightarrow\bar x$ as $k\to\infty$.

\item If $f$ satisfies the KL property at some accumulation point $\bar x$ of $\set{x^k}$ with $\psi(t)=Mt^{q}$ for $M>0$ and $q\in(0,1)$, then the following convergence rates are guaranteed:
\begin{itemize}

\item For $q\in(0,1/2]$, the sequences $\set{x^k}$, $\{f(x^k)\}$, and $\{\nabla f(x^k)\}$ converge linearly to $\bar x$, $f(\bar x)$, and $0\in\R^n$ respectively.

\item For $q\in(1/2,1)$, we have the convergence rate estimates
\begin{equation}\label{rate1}
\norm{x^k-\bar x}=\mathcal{O}\brac{k^{-\frac{1-q}{2q-1}}},
\end{equation}
\begin{equation}\label{rate2}
f(x^k)-f(\bar x)=\mathcal{O}\brac{k^{-\frac{2-2q}{2q-1}}},
\end{equation}
\begin{equation}\label{rate3}
\norm{\nabla f(x^k)}=\mathcal{O}\brac{k^{-\frac{1-q}{2q-1}}}.
\end{equation}
\end{itemize}
\end{enumerate}
\end{Theorem}
\begin{proof}
Since $\nabla f$ is Lipschitz continuous with constant $L$, we deduce from the descent condition (\ref{descent condition}), the relationship $x^{k+1}=x^k-\frac{1}{L}g^k$, and the condition \eqref{norm condition 2} that
\begin{align}
f(x^{k+1})&\le f(x^k)+\dotproduct{\nabla f(x^k),x^{k+1}-x^k}+\dfrac{L}{2} \norm{x^{k+1}-x^k}^2\nonumber\\
&= f(x^k)+\dotproduct{\nabla f(x^k)-g^k,-\dfrac{1}{L}g^k}-\dfrac{1}{L}\norm{g^k}^2+\dfrac{L}{2} \norm{\dfrac{1}{L}g^k}^2\nonumber\\
&\le f(x^k)+\dfrac{1}{L}\norm{\nabla f(x^k)-g^k}\cdot\norm{g^k}-\dfrac{1}{2L}\norm{g^k}^2\nonumber\\
&\le f(x^k)+\dfrac{\nu_2}{L}\norm{g^k}^2-\dfrac{1}{2L}\norm{g^k}^2 \nonumber\\
&= f(x^k)-\frac{1-2\nu_2}{2L}\norm{g^k}^2\;\text{ for all }\;k\in \N\label{descent constant}.
\end{align}
It follows from $\nu_2<\frac{1}{2}$ that the sequence $\set{f(x^k)}$ is decreasing. If $\norm{x^k}\rightarrow\infty$ as $k\to\infty$, there is nothing to prove, and thus we suppose that $\set{x^k}$ has an accumulation point $\bar x$. Then $f(\bar x)$ is an accumulation point of $\set{f(x^k)}$. Combining this with the decreasing property of $\set{f(x^k)}$, we deduce that $\lim_{k\in\N} f(x^k)=f(\bar x)$, which implies that  $f(x^k)-f(x^{k+1})\rightarrow 0$. Then \eqref{descent constant} tells us that
 $g^k\rightarrow0$, which being combined with \eqref{norm condition 2} gives us $\nabla f(x^k)\rightarrow0$ as $k\to\infty$ and hence verifies (i).

To justify the assertions in (ii) and (iii), let $\bar x$ be some accumulation point of $\set{x^k}$, and let $f$ satisfy the KL property at $\bar x$. It follows from \eqref{norm condition 2} that 
\begin{align}\label{rate condition}
\norm{\nabla f(x^k)}\le \norm{g^k}+\norm{\nabla f(x^k)-g^k}\le (1+\nu_2)\norm{g^k}\;\text{ for all }\;k\in\N.
\end{align}
Combining the latter with \eqref{descent constant} and the recurrent relation $x^{k+1}=x^k-\frac{1}{L}g^k$ gives us the estimates
\begin{align}
f(x^{k})-f(x^{k+1})&\ge \frac{1-2\nu_2}{2L}\norm{g^k}^2=\frac{1-2\nu_2}{2}\norm{g^k}\norm{\dfrac{1}{L}g^k}\label{5 15 1}\\ 
&\ge\frac{1-2\nu_2}{2(1+\nu_2)}\norm{\nabla f(x^k)}\cdot\norm{x^{k+1}-x^k}\;\text{ for all }\;k\in\N.\label{5.13 2}
\end{align}
Thus assumption (H1) in Proposition~\ref{general convergence under KL} is satisfied with $\sigma=\frac{1-2\nu_2}{2(1+\nu_2)}>0$. The second assumption (H2) in Proposition~\ref{general convergence under KL} is also satisfied since $f(x^k)=f(x^{k+1})$ yields $g^k=0$ by \eqref{5 15 1} and hence $x^{k+1}=x^k$ by the relationship $x^{k+1}=x^k-\frac{1}{L}g^k$. Thus Proposition~\ref{general convergence under KL} brings us to $x^k\rightarrow\bar x$ as $k\to\infty$, which verifies (ii). 

 To justify (iii), we first employ Proposition~\ref{general rate} with $d^k=-g^k$ for all $k\in\N$ and $\tau=\frac{1}{L}$. By \eqref{5 15 1} and \eqref{rate condition}, both conditions in \eqref{two conditions} are satisfied with $\alpha=\frac{1-2\nu_2}{2L}$ and $\beta=1+\nu_2$. It follows from \eqref{rate condition} and $\nabla f(x^k)\ne 0$ that $g^k\ne 0$ for all $k\in\N$. Combining this with $x^{k+1}=x^k-\frac{1}{L}g^k$, we get that $x^{k+1}\ne x^k$ for all $k\in\N$. Therefore, all the assumptions in Proposition~\ref{general rate} are satisfied, which verifies the convergence rate of $\set{x^k}$ to $\bar x$ stated in (iii). Since $\bar x$ is an accumulation point of $\set{x^k}$, we also deduce from (i) that $\bar x$ is a stationary point of $f$, i.e., $\nabla f(\bar x)=0$. 
Therefore, it follows from the descent condition \eqref{descent condition} and the decreasing property of $\set{f(x^k)}$ that
\begin{align*}
0\le f(x^k)-f(\bar x)\le \dotproduct{\nabla f(\bar x),x^k-\bar x}+\dfrac{{L}}{2}\norm{x^k-\bar x}^2=\dfrac{{L}}{2}\norm{x^k-\bar x}^2,
\end{align*}
which justifies the convergence rates of $\set{f(x^k)}$ to $f(\bar x)$ as asserted in (iii).

It remains to verify the convergence rates for $\set{\nabla f(x^k)}$.  Since $\nabla f$ is Lipschitz continuous with constant $L>0$, this follows from the convergence rates for $\set{x^k}$ due to
\begin{align*}
\norm{\nabla f(x^k)}=\norm{\nabla f(x^k)-\nabla f(\bar x)}\le L\norm{x^k-\bar x},
\end{align*}
which therefore completes the proof of the theorem.
\end{proof}

\begin{Remark}\rm 
By the triangle inequality, the estimate in \eqref{norm condition 1} with $\nu_1<\frac{1}{3}$ yields the one in \eqref{norm condition 2} with $\nu_2<\frac{1}{2}$. Therefore, Theorem~\ref{theo norm con} also verifies the global convergence of the sequences in the general inexact gradient methods under \eqref{norm condition 1}.
\end{Remark}

Now we are ready to establish the convergence properties of the proposed IGD Algorithm~\ref{al gd inexact}.

\begin{Theorem} \label{convergence constant}
Consider Algorithm~{\rm\ref{al gd inexact}} with $\mu>2$ and assume that $\nabla f(x^k)\ne 0$ for all $k\in\N$. Then we have that either $\norm{x^k}\rightarrow\infty$, or the assertions in {\rm(i)--(iii)} in Theorem ~{\rm\ref{theo norm con}} hold with $\varepsilon_k\downarrow0$ as $k\to\infty$.
\end{Theorem}
\begin{proof}
It follows from \eqref{epsilon k+1} that condition \eqref{norm condition 2} is satisfied with $\nu_2=\frac{1}{\mu}<2$. Then applying Theorem~\ref{theo norm con} yields its assertions (i)--(iii). The convergence to $0$ of $\set{g^k}$ in (i) and the inequality $\norm{g^k}>\mu\varepsilon_{k+1}$ for all $k\in\N$ from \eqref{epsilon k+1} tells us that $\varepsilon_k\downarrow 0$ and thus completes the proof.
\end{proof}

\section{Gradient-Based Inexact Methods in Nonsmooth Convex Optimization}\label{sec:5}

In this section, we employ the inexact gradient method for ${\cal C}^{1,1}$ optimization developed in Section~\ref{sec Stationarity of accumulation points} to design and justify two new gradient-based \textit{inexact methods} to solve problems of {\em nonsmooth convex optimization}. The reductions of such nonsmooth problems to  ${\cal C}^{1,1}$ optimization is conducted by using {\em smoothing procedures} via Moreau envelopes and proximal mappings.\vspace*{0.05in}

The methods we develop in this way are the following:

\begin{enumerate}[\bf(i)]
\item The {\em gradient-based inexact proximal point method} (GIPPM) to minimize l.s.c. convex functions.

\item The {\em gradient-based  inexact augmented Lagrangian method} (GIALM) to solve nonsmooth convex programs with equality linear constraints.
\end{enumerate}

Each of these two methods is considered in the corresponding subsection below. First we recall the appropriate constructions of convex analysis used in what follows; see \cite{bauschkebook,mor-nam,rockafellarbook} for more details. Given a proper (i.e., with $\dom g:=\{x\in\R^n\;|\;g(x)<\infty\}\ne\emp$), l.s.c., convex function $g \colon\R^n\to\oR:=(-\infty,\infty]$, the {\em subdifferential} of $g$ at $x\in\dom g$ is defined by
\begin{equation*}
\partial g(x):=\big\{v\in\R^n\;\big|\;\dotproduct{v,y-x}\le g(y)-g(x)\;\text{ for all }\;y\in\R^n\big\}.
\end{equation*}
For any proximal parameter $\lambda >0$, the {\it Moreau envelope} $e_{\lambda}g \colon\R^n\rightarrow\R$ and the {\it proximal mapping} $\textit{\rm Prox}_{\lambda g}\colon\R^n\rightrightarrows\R^n$ are defined, respectively, by
\begin{equation}\label{Moreau}
\begin{array}{ll}
e_{\lambda}g (x):=\inf_{y \in \R^n}\set{g (y)+\dfrac{1}{2\lambda}\norm{y-x}^2},\\\prox_{\lambda g }(x):=\mathop{\rm argmin}_{y\in \R^n}\set{g (y)+\dfrac{1}{2\lambda}\norm{y-x}^2}.
\end{array}
\end{equation}

The following result taken from \cite[Proposition 12.28, 12.29]{bauschkebook} and \cite[Theorem~2.26, Proposition~12.19]{rockafellarbook} presents remarkable properties of Moreau envelopes and proximal mappings for convex functions that allow us to pass from convex nonsmooth optimization problems with extended-valued objectives (i.e., incorporating constraints) to problems of ${\cal C}^{1,1}$ optimization.

\begin{Lemma}\label{rela} Let $g\colon\R^n\to\oR$ be a proper, l.s.c., convex function, and let $\lambda>0$. Then the Moreau envelope $e_\lambda g $ is convex, $\mathcal{C}^1$-smooth, and has the Lipschitz continuous gradient on $\R^n$ with constant $1/\lambda$ that satisfies the relationship
\begin{equation}\label{gradientMoreau}
\nabla e_\lambda g (x)=\lambda^{-1}\Big(x-\text{\rm Prox}_{\lambda g }(x)\Big)\;\mbox{ for all }\;x\in\R^n.
\end{equation}
Furthermore, the proximal mapping $\prox_{\lambda g}$ is Lipschitz  continuous on $\R^n$ with constant $1$, and every stationary point of $e_\lambda g$ is a minimizer of $g$.
\end{Lemma}

\subsection{Gradient-Based Inexact Proximal Point Method}

We consider here the class of optimization problems given in the form
\begin{align}\label{minimizing convex function}
{\rm minimize}\quad g(x)\;\text{ subject to }\; x\in\R^n,
\end{align}
where $g:\R^n\rightarrow\oR$ is a proper, l.s.c., and convex function. Although problems \eqref{optim prob} and \eqref{minimizing convex function} are written in the similar format, there is a huge difference between the ${\cal C}^{1,1}$ objective in \eqref{optim prob} and the extended-real-valued objective in
\eqref{minimizing convex function}. Besides different differential properties of the objective functions in \eqref{optim prob} and \eqref{minimizing convex function}, observe that \eqref{optim prob} is a problem of unconstrained optimization, while \eqref{minimizing convex function} incorporates constraints coming from $x\in\dom g$ when the function $g$ is extended-real-valued.\vspace*{0.05in} 

The following gradient-based inexact proximal point method/algorithm GIPPM is now proposed to solve the general class of nonsmooth convex optimization problems \eqref{minimizing convex function}.\\

\begin{longfbox}
\begin{Algorithm}[GIPPM]\hlabel{ppm}
\setcounter{Step}{-1}
\begin{Step}[initialization]\rm
Choose a proximal parameter $\lambda >0$, initial point $x^1\in \mathbb{R}^n$, initial error $\varepsilon_1>0$, error reduction factor $\theta\in(0,1)$, and  scaling factor $\mu>2$. Set $k:=1$.
\end{Step}
\begin{Step}[inexact proximal mapping]\rm\label{step 1 ppm} 
Find $p^k$ and the smallest natural number $i_k$ such that
\begin{align}\label{inexact proximal mapping}
\norm{p^k-\prox_{\lambda g}(x^k)}\le \lambda\theta^{i_k}\varepsilon_k\quad \text{ and }\quad \norm{x^k-p^k}>\lambda\mu \theta^{i_k}\varepsilon_k.
\end{align}
\end{Step}
\begin{Step}[error and iteration update]\rm\label{step 2 IPPMm}
Set $x^{k+1}:=p^k$, $\varepsilon_{k+1}:= \theta^{i_k}\varepsilon_k$. Increase $k$ by $1$ and go to Step~\ref{step 1 ppm}.
\end{Step}
\end{Algorithm}
\end{longfbox}

\begin{Remark}\rm
Similarly to the discussions in Remark~\ref{remark IGD}, the procedure of finding $p^k$ and $i_k$ that satisfy \eqref{inexact proximal mapping} can be given as follows. Set $i_k=0.$ Find some $p^k$ such that 
\begin{align*}
\norm{p^k-\prox_{\lambda g}(x^k)}\le\lambda \theta^{i_k}\varepsilon_k.
\end{align*}
While $\norm{x^k-p^k}\le\lambda \mu\theta^{i_k}\varepsilon_k$, increase $i_k$ by $1$ and recalculate $p^k$ under the condition above. Then the existence of $p^k$ and $i_k$ satisfying \eqref{inexact proximal mapping} is guaranteed when $0\notin \partial g(x^k)$. Indeed, otherwise for each $k\in\N$ we get a sequence $\set{p^k_i}$ as $i\in\N$ with
\begin{align*}
\norm{p^k_i-\prox_{\lambda g}(x^k)}\le \lambda\theta^i\varepsilon_k\text{ and }\norm{x^k-p^k_i}\le \lambda\mu \theta^i\varepsilon_k \text{ for all }i\in\N.
\end{align*}
Since $\theta\in(0,1),$ the latter inequalities mean that $x^k=\prox_{\lambda g}(x^k)$, which is equivalent to
\begin{align*}
x^k=\mathop{\rm argmin}_{x\in\R^n}\set{g(x^k)+\dfrac{1}{2\lambda}\norm{x-x^k}^2}.
\end{align*}
By using the subdifferential Fermat rule, we get $0\in\partial g(x^k),$ a contradiction.
\end{Remark}

Now we are ready to establish convergence properties of Algorithm~\ref{ppm}.

\begin{Theorem}\label{theorem ppm}
Let $\set{x^k}$ be the iterative sequence generated by Algorithm~{\rm\ref{ppm}} with $\mu>2$. Assume that $0\notin\partial g(x^k)$ for all $k\in\N.$ Then either $\norm{x^k}\rightarrow\infty$ as $k\to\infty$, or we have the assertions:

\begin{enumerate}[\bf(i)]
\item Every accumulation point of $\set{x^k}$ is a minimizer of $g$, and both sequences $\set{\norm{x^k-p^k}} and \set{\varepsilon_k}$ converge to $0$ as $k\to\infty$.

\item If $e_\lambda g$ satisfies the KL property at an accumulation point $\bar x$ of $\set{x^k}$, then $x^k\rightarrow\bar x$ as $k\to\infty$.

\item If $e_\lambda g$ satisfies the KL property at an accumulation point $\bar x$ of $\set{x^k}$ with $\psi(t)=Mt^{q}$ for some $M>0$ and $q\in(0,1)$, then  the following convergence rates are guaranteed as $k\to\infty$: 

$\bullet$ For $q\in(0,1/2]$, the sequence $\set{x^k}$ converges linearly to $\bar x$ as $k\to\infty$.

$\bullet$ For $q\in(1/2,1)$, we have $\norm{x^k-\bar x}=\mathcal{O}\big(k^{-\frac{1-q}{2q-1}}\big)$ as $k\to\infty$.
\end{enumerate}
\end{Theorem}
\begin{proof}
Define $f:=e_\lambda g$. By Lemma~\ref{rela}, we have that $\nabla f$ is Lipschitzian with constant $L=1/\lambda$ and 
\begin{align*}
\nabla f(x)=\lambda^{-1}\big(x-\prox_{\lambda g}(x)\big)\;\mbox{ for all }\;x\in\R^n.
\end{align*}
Defining $g^k:=\lambda^{-1}(x^k-p^k)$, the inequalities in \eqref{inexact proximal mapping} can be rewritten as
\begin{align*}
\norm{g^k-
\nabla f(x^k)}&=\norm{g^k-
\nabla e_\lambda g(x^k)}\\
&=
\norm{ \lambda^{-1}(x^k-p^k)-
\lambda^{-1}(x^k-{\rm Prox}_{\lambda g}(x^k))}\\
&=\lambda^{-1}\norm{{\rm Prox}_{\lambda g}(x^k)-p^k}\le\theta^{i_k} \varepsilon_{k}\;\mbox{ and}
\end{align*}
\begin{align*}
\norm{g^k}=\lambda^{-1}\norm{x^k-p^k}>\mu\theta^{i_k}\varepsilon_k,
\end{align*}
respectively.
Furthermore, it follows from Step~\ref{step 2 IPPMm} of Algorithm~\ref{ppm} that the iterative procedure in this algorithm can be represented by
\begin{align*}
x^{k+1}=p^k=x^k-\lambda g^k=x^k-\dfrac{1}{L}g^k,
\end{align*}
Therefore, $\set{x^k}$ is the iterative sequence generated by Algorithm~\ref{al gd inexact} with $f=e_\lambda g$. Then Theorem~\ref{convergence constant}(i) tells us that every accumulation point of $\set{x^k}$ is a stationary point of $e_\lambda g$, which is actually a minimizer of $g$ by Lemma~\ref{rela}. Finally, assertions (ii) and (iii) follow from Theorem~\ref{convergence constant}(ii), (iii), respectively.
\end{proof}

Next we discuss some features of Algorithm~\ref{ppm} and its relationships with other developments.

\begin{Remark}\rm Observe the following:
\begin{enumerate}[\bf (i)]
 \item By using the properties of semialgebraic functions listed in Remark~\ref{algebraic} and the Moreau envelope construction in \eqref{Moreau}, we can verify that the assumption on the KL property of $e_\lambda g$ in Theorem~\ref{convergence constant} is satisfied if $g$ is a semialgebraic function. 
 
 \item Applying Theorem~\ref{theo norm con} and using the arguments similar to Theorem~\ref{theorem ppm} allow us to get the convergence properties for Algorithm~\ref{ppm} if \eqref{inexact proximal mapping} is replaced by more general conditions
\begin{align*}
 \norm{p^k-\prox_{\lambda g}(x^k)}\le \nu_2\norm{x^k-p^k}\;\text{ with }\;\nu_2<\frac{1}{2},
\end{align*}
\begin{align*}
\norm{p^k-\prox_{\lambda g}(x^k)}\le \nu_1\norm{x^k-\prox_{\lambda g}(x^k)}\text{ with }\nu_1<\frac{1}{3}.
\end{align*}

\item Let us highlight important improvements of our GIPPM over the two versions of the classical inexact proximal point method (IPPM) of Rockafellar \cite{rockafellar76} applied to \eqref{minimize proximal} as follows:
\begin{align}
&x^{k+1}=p^k,\;\norm{p^k-\prox_{\lambda g}(x^k)}\le\delta_k,\; \sum_{k=1}^\infty\delta_k<\infty,\tag{A}\label{predetermined error}\\
&x^{k+1}=p^k,\;\norm{p^k-\prox_{\lambda g}(x^k)}\le \delta_k\norm{x^k-p^k},\;\sum_{k=1}^\infty\delta_k<\infty.\tag{B}\label{predetermined error 2}
\end{align}
To this end, we mentioned first that in deriving convergence results, our approach in Algorithm \ref{ppm} is by employing the IGD method applied to the Moreau envelope $e_\lambda g$, while the one by Rockafellar in \eqref{predetermined error} and \eqref{predetermined error 2} is using the properties related to the maximal monotonicity of the subgradient mapping $\partial g$. This creates various differences between our GIPPM and Rockafellar's IPPM. In comparison with \eqref{predetermined error} of IPPM, our GIPPM can achieve a linear convergence rate while the convergence rate of \eqref{predetermined error} is not established in \cite{rockafellar76}. In practice, the sequence of errors $\set{\delta_k}$ in \eqref{predetermined error} is usually chosen as $\delta_k:=ck^{-p}$ with $p>1$ and $c>0$. If $\delta_k$ is too small, the procedure of finding approximate proximal points $p^k$ takes a lot of time while if $\delta_k$ is too large, $p^k$ becomes unreliable. Note that in each iteration, the GIPPM algorithm verifies whether the approximate proximal point is acceptable and then decreases the error if necessary.

Regarding the version \eqref{predetermined error 2} of IPPM, it follows from the choice of $\delta_k$ that $\delta_k\rightarrow0$ as $k\to\infty$. This {ensures} that for any $\nu_2<\frac{1}{2}$ there is $N\in\N$ with $\delta_k<\nu_2$, which gives us
\begin{align*}
\norm{p^k-\prox_{\lambda g}(x^k)}\le \nu_2\norm{x^k-p^k}\;\text{ as }\;k\ge N. 
\end{align*}
Therefore, the convergence properties of the algorithm with the procedure \eqref{predetermined error 2} can be deduced from (ii). In addition to the advantage on the generality in convergence analysis, our GIPPM method is also promising in practical implementation. Indeed, since $\nu_2$ is a constant while $\delta_k$ is decreasing to $0$, the procedure of finding $p^k$ under the conditions in \eqref{inexact proximal mapping} of Algorithm~\ref{ppm} takes significantly less time than the procedure \eqref{predetermined error 2} of IPPM while maintaining impressive convergence results achieved in Theorem~\ref{theorem ppm}.
\end{enumerate}
\end{Remark}

\subsection{Gradient-Based Inexact Augmented Lagrangian Method}

In this subsection, we propose the gradient-based inexact  augmented Lagrangian method/algorithm GIALM to solve equality-constrained convex programs given in the form:
\begin{equation}
\tag{P} \label{main problem}
\begin{split}
&{\rm minimize}\quad h(x)\\ 
 &{\rm subject\  to}\quad Ax=b, 
\end{split}
\end{equation}
where $h\colon\mathbb{R}^n\rightarrow\overline{\mathbb{R}}$ is a proper, l.s.c., convex function, and where $A\in\mathbb{R}^{m\times n},\;b\in\mathbb{R}^m$. Assume in what follows that the feasible set $\set{x\in\R^n\;|\;Ax=b}$ of \eqref{main problem} is nonempty. Given a parameter $\lambda>0$, for each pair $(x, y )\in\mathbb{R}^n\times \mathbb{R}^m$, consider the {\em Lagrange function} and the {\em augmented Lagrangian} defined as in \cite{hestenes69}, respectively, by
\begin{align}\label{lagrange}
\ell(x,y):=h(x)+\dotproduct{ y ,Ax-b}\;\text{ and }\; \mathcal{L}_\lambda(x, y):&=\ell(x,y)+\dfrac{\lambda}{2}\norm{Ax-b}^2.
\end{align}
The {\em dual Lagrange function} $d\colon\R^m\rightarrow [-\infty,\infty)$ is given by $d(y ):=\inf_{x\in\mathbb{R}^n} \ell(x,y)$, and the corresponding {\em dual problem} is formulated as
\begin{equation}
\label{dual problem alm}\tag{D}
\begin{split}
&{\rm maximize}\quad d(y)\\
&{\rm subject\ to}\quad y\in\R^m.
\end{split}
\end{equation}
It follows from \cite[Lemma~9]{jonathan15} that the function $g:=-d$ is l.s.c.\ and convex. Therefore, when $g$ is proper, its proximal mapping is well-defined and can be represented via the dual function $d$ as
\begin{align}\label{proxalm}
\prox_{\lambda g}(y)=\mathop{\rm argmin}_{w\in\R^m}\set{g(w)+\dfrac{1}{2\lambda}\norm{w-y}^2}=\mathop{\rm argmax}_{w\in\R^m}\set{d(w)-\dfrac{1}{2\lambda}\norm{w-y}^2}.
\end{align}

Now we present the construction of the gradient-based inexact augmented Lagrangian method.\\

\begin{longfbox}
 \begin{Algorithm}[GIALM]\rm\label{IAL}
\quad
\setcounter{Step}{-1}
\begin{Step}\rm
Choose a starting point $y ^1\in\mathbb{R}^m$, proximal parameter $\lambda>0$, initial error $\varepsilon_1>0$, error reduction factor $\theta \in (0,1)$, and scaling factor $\mu>2$. Set $k:=1.$ 
\end{Step}
\begin{Step}\rm Find some
$x^{k+1}$ and the smallest natural number $i_k$ such that
\begin{align}\label{subproblem}
\mathcal{L}_{\lambda}(x^{k+1},y^k)-\inf_{x\in\R^n} \mathcal{L}_{\lambda}(x,y^k)\le \dfrac{\lambda\theta^{2i_k}\varepsilon_k^2}{2}\;\text{ and }\;\norm{b-Ax^{k+1}}>\mu \theta^{i_k}\varepsilon_k.
\end{align}
\end{Step}
\begin{Step}\rm
Set $y^{k+1}:=y^k+\lambda(Ax^{k+1}-b)$ and $\varepsilon_{k+1}:=\theta^{i_k}\varepsilon_k$. Increase $k$ by $1$ and go back to Step 1.
\end{Step}
\end{Algorithm}
\end{longfbox}\vspace*{0.2in}

As a part of our convergence analysis for GIALM, we derive the following two propositions of their independent interest. The first proposition provides a useful representation of the augmented Lagrangian function different from \eqref{lagrange}.

\begin{Proposition}\label{two lagrange rela}
For any $z\in \mathbb{R}^n$, $\eta \in \mathbb{R}^m$, and $\lambda>0$, we have the representation
\begin{align*}
\mathcal{L}_\lambda(z,\eta)=\sup_{w\in\R^n}\set{\ell(z,w)-\dfrac{1}{2\lambda}\norm{w-\eta}^2}.
\end{align*}
\end{Proposition}
\begin{proof}
Taking into account the construction of $\ell$ in \eqref{lagrange}, we see that 
\begin{align*}
w\mapsto\ell(z,w)-\dfrac{1}{2\lambda}\norm{w-\eta}^2=h(z)+\dotproduct{w,Az-b}-\dfrac{1}{2\lambda}\norm{w-\eta}^2
\end{align*}
is a quadratic concave function, and thus it attains the global maximum when $Az-b=\frac{1}{\lambda}(w-\eta)$, i.e., at $w=\lambda(Az-b)+\eta$. Thus we have
\begin{align*}
\sup_w\set{\ell (z,w)-\dfrac{1}{2\lambda}\norm{w-\eta}^2} =\ &h(z)+\big\la\lambda(Az-b)+\eta,Az-b\big\ra-\dfrac{1}{2\lambda}\norm{\lambda(Az-b)}^2\\
=\ &h(z)+\dotproduct{\eta,Az-b}+\dfrac{\lambda}{2}\norm{Az-b}^2\\
=\ &\mathcal{L}_\lambda(z,\eta),
\end{align*}
which therefore verifies the claim of the proposition.
\end{proof}

Next we obtain a relationship between the augmented Lagrangian \eqref{lagrange} and the proximal mapping \eqref{proxalm}. The result of this type was stated in \cite[Proposition~6]{rockafellar alm} but only for the iterative sequence of the inexact augmented Lagrangian method for convex programs with inequality constraints. Using a similar technique, a general result for problem \eqref{main problem} without considering a specific iterative sequence is presented in the next proposition.

\begin{Proposition}\label{theo rockafellar}
Suppose that the function $g$ defined above is proper. Then we have
\begin{align*}
\dfrac{1}{2\lambda}\norm{y+\lambda(Ax-b)-\prox_{\lambda g}(y)}^2 \le  \mathcal{L}_{\lambda}(x,y)-\inf_{z\in\R^n} \mathcal{L}_{\lambda}(z,y)
\end{align*}
for all pairs  $(x,y)\in\R^n\times \R^m$.
\end{Proposition}
\begin{proof}
Fix any $(x,y)\in\R^n\times\R^m$ and denote $p=y+\lambda(Ax-b)$. Take an arbitrary $\eta\in\R^m$ and define the function $\vt\colon\R^n\times \R^m\rightarrow\R$ by
\begin{align*}
\vt(z,w):=\ell(z,w)-\dfrac{1}{2\lambda}\norm{w-\eta}^2=  h(z)+\dotproduct{w,Az-b}-\dfrac{1}{2\lambda}\norm{w-\eta}^2.
\end{align*}
It is clear that the function $\vt$ is l.s.c.\ convex in $z$, and continuous concave in $w$ with $\vt(0,w)\rightarrow-\infty$ as $\norm{w}\rightarrow\infty$. Applying Proposition~\ref{two lagrange rela} and then \cite[Theorem~2.7]{tuybook} to $\vt(z,w)$ with taking \cite[Remark~2.2]{tuybook} into account, we get the equalities
\begin{align}\label{using twice}
\inf_z \mathcal{L}_\lambda(z,\eta)&= \inf_z \sup_w \set{\ell(z,w)-\dfrac{1}{2\lambda}\norm{w-\eta}^2}\nonumber\\
&=\sup_w \inf _z \set{\ell(z,w)-\dfrac{1}{2\lambda}\norm{w-\eta}^2}\nonumber\\
 &=\sup_w \set{d(w)-\dfrac{1}{2\lambda}\norm{w-\eta}^2}.
\end{align}
Combining this with the construction of $\mathcal{L}_\lambda(x,y)$ in \eqref{lagrange} gives us
\begin{align}\label{dprox}
\mathcal{L}_\lambda(x,y)+\dfrac{1}{\lambda}\dotproduct{p-y,\eta-y}&=h(x)+\dotproduct{y,Ax-b}+\dfrac{\lambda}{2}\norm{Ax-b}^2+\dotproduct{Ax-b,\eta-y}\nonumber\\
&=h(x)+\dotproduct{\eta,Ax-b}+\dfrac{\lambda}{2}\norm{Ax-b}^2\nonumber\\
&=\mathcal{L}_\lambda (x,\eta)\ge \inf_z \mathcal{L}_\lambda(z,\eta)\nonumber\\
&= \sup_{w}\set{d(w)-\dfrac{1}{2\lambda}\norm{w-\eta}^2}\nonumber\\
&\ge d\big(\prox_{\lambda g}(y)\big)-\dfrac{1}{2\lambda}\norm{\prox_{\lambda g}(y)-\eta}^2.
\end{align}
Employing \eqref{using twice} again with $\eta:=y$ and taking into account the construction of the proximal mapping in \eqref{proxalm}, we arrive at the equalities
\begin{align*}
\inf_z \mathcal{L}_\lambda(z,y)&=\sup_w \set{d(w)-\dfrac{1}{2\lambda}\norm{w-y}^2}\\
&=d\big(\prox_{\lambda g}(y))\big)-\dfrac{1}{2\lambda}\norm{\prox_{\lambda g}(y)-y}^2.
\end{align*}
The latter expression combined with \eqref{dprox} tells us that
\begin{align*}
\mathcal{L}_\lambda(x,y)-\inf_z \mathcal{L}_\lambda(z,y)&\ge \dfrac{1}{2\lambda}\brac{\norm{\prox_{\lambda g}(y)-y}^2-\norm{\prox_{\lambda g}(y)-\eta}^2-2\dotproduct{p-y,\eta-y}}.
\end{align*}
Choosing finally $\eta:=\prox_{\lambda g}(y)+y-p$ brings us to
\begin{align*}
\mathcal{L}_\lambda(x,y)-\inf_z \mathcal{L}_\lambda(z,y)&\ge \dfrac{1}{2\lambda}\brac{\norm{\prox_{\lambda g}(y)-y}^2-\norm{p-y}^2-2\dotproduct{p-y,\prox_{\lambda g}(y)-p}}\\
 &=\dfrac{1}{2\lambda}\norm{p-\prox_{\lambda g}(y)}^2,
\end{align*}
which therefore completes the proof of the proposition.
\end{proof}

\begin{Remark}\rm
The procedure of finding $x^{k+1}$ and $i_k$ in Step~1 of Algorithm~\ref{IAL} can be performed similarly to Algorithm~\ref{al gd inexact} and Algorithm~\ref{ppm} in the following way. Set $i_k=0$ and find some $x^{k+1}$ with
\begin{align}\label{find xk+1 gk}
\mathcal{L}_{\lambda}(x^{k+1},y^k)-\inf_{x\in\R^n} \mathcal{L}_{\lambda}(x,y^k)\le \dfrac{\lambda\theta^{2i_k}\varepsilon_k^2}{2}.
\end{align}
While $\norm{b-Ax^{k+1}}\le \mu \theta^{i_k}\varepsilon_k,$ increase $i_k$ by $1$ and recalculate $x^{k+1}$ under condition \eqref{find xk+1 gk}. This procedure terminates finitely when the function $g$ defined above is proper, and when $y^k$ is not a solution to \eqref{dual problem alm}. Indeed, if \eqref{find xk+1 gk} does not stop, for each $k\in\N$ we get $\set{x^{k+1}_i}$, $i\in\N$, such that
\begin{align}\label{two alm}
\mathcal{L}_{\lambda}(x^{k+1}_i,y^k)-\inf_{x\in\R^n} \mathcal{L}_{\lambda}(x,y^k)\le \dfrac{\lambda\theta^{2i}\varepsilon_k^2}{2} \quad \text{ and }\quad\norm{b-Ax^{k+1}_i}\le \mu \theta^{i}\varepsilon\;\text{ for all }\;i\in\N.
\end{align}
Combining the first inequality in \eqref{two alm} with Proposition~\ref{theo rockafellar} gives us
\begin{align*}
\dfrac{1}{2\lambda}\norm{y^k+\lambda(Ax^{k+1}_i-b)-\prox_{\lambda g}(y^k)}^2 \le \dfrac{\lambda\theta^{2i}\varepsilon_k^2}{2}.
\end{align*}
Since $\theta\in (0,1)$, letting $i\rightarrow\infty$ with taking into account the second inequality in \eqref{two alm}, we get that  $y^k=\prox_{\lambda g}(y^k)$, which implies by the subdifferential Fermat rule that $0\in\partial g(y^k)$. Therefore, $y^k$ is a solution to the dual problem \eqref{dual problem alm} due to its convexity. This is a contradiction. 
\end{Remark}

Now we are ready to establish the convergence results of our GIALM algorithm.

\begin{Theorem}\label{theo alm}
Let the sequences $\set{x^k}$, $\set{y^k}$ are generated by Algorithm~{\rm\ref{IAL}}. Assume that $\sup_{y\in\R^m} d(y)>-\infty$ and $y^k$ is not a solution to \eqref{dual problem alm} whenever $k\in\N$. Then either $\norm{y^k}\rightarrow\infty$ as $k\to\infty$, or we have:
\begin{enumerate}
[\bf (i)]
\item Every accumulation point of $\set{y^k}$ is a solution to \eqref{dual problem alm}. If in addition the sequence $\set{y^k}$ is bounded, then every accumulation point of $\set{x^k}$ is a solution to \eqref{main problem}.

\item If $e_\lambda g$ satisfies the KL property at an accumulation point $\bar y$ of $\set{y^k}$, then $y^k\rightarrow\bar y$ as $k\to\infty$.

\item If $e_\lambda g$ satisfies the KL property at an accumulation point {$\bar y$ of $\set{y^k}$} with $\psi(t)=Mt^{q}$ for some $M>0$ and $q\in (0,1)$, then we get the convergence rates:

$\bullet$ If $q\in(0,1/2]$, then the sequence $\set{y^k}$ converges linearly to $\bar y$ as $k\to\infty$.

$\bullet$ If $q\in(1/2,1)$, then $\|y^k-\bar y\|={\cal O}\big(k^{-\frac{1-q}{2q-1}}\big)$ as $k\to\infty$.
\end{enumerate}
\end{Theorem}
\begin{proof}
(i) Since $\sup_{y\in\R^m}d(y)>-\infty$, the function $g=-d$ in \eqref{dual problem alm} is proper. It follows from Proposition~\ref{theo rockafellar} and the first inequality in \eqref{subproblem} that
{\begin{align*}
\dfrac{1}{2\lambda}\norm{y^k+\lambda(Ax^{k+1}-b)-\prox_{\lambda g}(y^k)}^2\le\dfrac{\lambda\theta^{2i_k}\varepsilon_k^2}{2}.
\end{align*}}
Defining $p^k:=y^k+\lambda(Ax^{k+1}-b)$, the latter can be equivalently written as
\begin{align*}
\norm{p^k-\prox_{\lambda g}(y^k)}\le \lambda \theta^{i_k}\varepsilon_k\;\text{ for all }\;k\in\N.
\end{align*}
Moreover, the second inequality in \eqref{subproblem} also tells us that
{\begin{align*}
\norm{y^k-p^k}>\lambda\mu \theta^{i_k}\varepsilon_k\;\text{ for all }\;k\in\N.
\end{align*}}
Therefore, $\set{y^k}$ is the iterative sequence generated by Algorithm \ref{ppm} to find minimizers of $g$. Then Theorem \ref{theorem ppm}(i) tells us that every accumulation point of $\set{y^k}$ is a minimizer of $g$, i.e., it is a solution to \eqref{dual problem alm}. Furthermore, we have $\varepsilon_k\downarrow0$ and $\norm{y^k-p^k}\rightarrow 0$ as $k\rightarrow\infty$. 

Assume next that $\set{y^k}$ is bounded and show that every accumulation point of $\set{x^k}$ is a solution to \eqref{main problem}. Since $\norm{y^k-p^k}\rightarrow0$, we deduce from the construction of $p^k$ that
\begin{align}\label{feasible}
Ax^k-b\rightarrow0\;\text{ as }\;k\rightarrow\infty.
\end{align} 
Picking any feasible solution $x^*$ to \eqref{main problem} tells us that $Ax^*=b$. It follows from the first inequality in \eqref{subproblem} and Step~2 of Algorithm~\ref{IAL} that
\begin{align*}
\disp\mathcal{L}_\lambda(x^{k+1},y^k)\le \inf_{x\in\R^n}\mathcal{L}_\lambda (x,y^k)+\dfrac{\lambda\varepsilon_{k+1}^2}{2}\le \disp\mathcal{L}_\lambda (x^*,y^k)+\dfrac{\lambda\varepsilon_{k+1}^2}{2}.
\end{align*} 
Combining this with the construction of $\mathcal{L}_\lambda$ in \eqref{lagrange}  and the feasibility of $x^*$ gives us 
\begin{align}\label{alm ineq}
h(x^{k+1})+\dotproduct{y^{k},Ax^{k+1}-b}+\dfrac{\lambda}{2}\norm{Ax^{k+1}-b}^2&\le h(x^*)+\dotproduct{y^{k},Ax^*-b}+\dfrac{\lambda}{2}\norm{Ax^*-b}^2+\dfrac{\lambda}{2}\varepsilon_{k+1}^2\nonumber\\
&= h(x^*)+\dfrac{\lambda}{2}\varepsilon_{k+1}^2\;\text{ for all }\;k\in\N.
\end{align}
Taking now any accumulation point $\bar x$ of $\set{x^k}$, we find some infinite subset $J\subset\N$ such that $x^k\overset{J}{\rightarrow}\bar x$. Then it follows from \eqref{feasible} that $A\bar x=b$, which verifies the feasibility of $\bar x$. Invoking \eqref{feasible} again together with the boundedness of $\set{y^k}$ gives us the convergence
\begin{align}
\dotproduct{y^{k-1},Ax^k-b}+\dfrac{\lambda}{2}\norm{Ax^k-b}^2\rightarrow 0\;\text{ as }\;k\rightarrow\infty.
\end{align}
Combining the latter with the lower semicontinuity of $h$, the estimate \eqref{alm ineq}, and the condition $\varepsilon_k\downarrow0$ tells us that
\begin{align*}
h(\bar x)&\le \liminf_{k\overset{J}{\rightarrow}\infty} h(x^k)\\
&=\liminf_{k\overset{J}{\rightarrow}\infty}\set{h(x^k)+\dotproduct{y^{k-1},Ax^k-b}+\dfrac{\lambda}{2}\norm{Ax^k-b}^2}\\
&\le \liminf_{k\overset{J}{\rightarrow}\infty}\set{h(x^*)+\dfrac{\lambda}{2}\varepsilon_{k}^2}=h(x^*).
\end{align*}
Since $x^*$ is any feasible solution to \eqref{main problem}, we conclude that $\bar x$ is an optimal solution to \eqref{main problem}.

Finally, assertions (ii) and (iii) follow from Theorem~\ref{theorem ppm}(ii), (iii), respectively.
\end{proof}

\section{Numerical Experiments}\label{sec:6}

In this section, we compare numerical aspects of the newly designed GIALM and the classical IALM in solving the basic Lasso problem given by
\begin{align}\label{primal}
\mathop{\rm minimize}\quad \dfrac{1}{2}\norm{Ax-b}^2+\gamma \norm{x}_1\;\text{ subject to }\; x\in\R^n, 
\end{align}
where $A$ is an $m\times n$ matrix, $b$ is a vector in $\R^m$, and $\gamma>0$. By omitting constants and defining $p(x):=\gamma \norm{x}_1$ and $c:=A^*b$, we observe that Lasso problem \eqref{primal} is equivalent to 
\begin{align}\label{transprim}
\mbox{maximize }\;-\dfrac{1}{2}\norm{Ax}^2+\dotproduct{c,x}-p(x),\quad x\in\R^n
\end{align}
This is the dual problem of the following primal convex program with linear equality constraints:
\begin{align}\label{dual}
&{\rm minimize}\quad \dfrac{1}{2}\norm{y}^2+p^*(z)\\
&{\rm subject\ to}\quad -A^*y-z=-c,\nonumber
\end{align}
where the indicator function $p^*(\cdot)=\delta_{\gamma\mathbb{B}^\infty}(\cdot)$ is the convex conjugate of $p$ with $\mathbb{B}^\infty:=\set{z\in\R^n\;|\;\max |z_i|\le \gamma}$. Indeed, the corresponding Lagrange function and the dual function of \eqref{dual} are given, respectively, by
\begin{align*}
\ell(y,z,x)=\dfrac{1}{2}\norm{y}^2+p^*(z)+\dotproduct{x,c-A^*y-z},
\end{align*}
\begin{align*}
d(x)&=\inf_{y,z}\ell(y,z,x)=\inf_{y}\set{\dfrac{1}{2}\norm{y}^2-\dotproduct{y,Ax}}+\inf_{z}\big\{p^*(z)-\dotproduct{x,z}\big\}+\dotproduct{c,x}\\
&=-\dfrac{1}{2}\norm{Ax}^2+\dotproduct{c,x}-p^{**}(x).
\end{align*}
Since $p$ is convex and continuous, we know from basic convex analysis that $p^{**}=p$, and thus $d$ is exactly the objective function in \eqref{transprim}. Since \eqref{dual} has the form of \eqref{main problem}, we can apply GIALM to solve \eqref{dual}, which brings us to the following algorithm.\\

\begin{longfbox}
 \begin{Algorithm}[GIALM for solving \eqref{dual}]
\quad
\setcounter{Step}{-1}
\begin{Step}\rm
Choose a starting point $ x^1 \in\mathbb{R}^n$, proximal parameter $\lambda>0,$ initial error $\varepsilon_1>0$, error reduction factor $\theta \in (0,1)$,  and scaling factor $\mu>2$. Set $k:=1$. 
\end{Step}
\begin{Step}\rm Find some
$(y^{k+1},z^{k+1})$ and the smallest natural number $i_k$ such that
{\begin{align}\label{subproblem numerical}
\mathcal{L}_{\lambda}(y^{k+1},z^{k+1},x^k)-\inf_{y,z} \mathcal{L}_{\lambda}(y,z,x^k)\le \dfrac{\lambda\theta^{2i_k}\varepsilon_k^2}{2} \quad \text{ and } \norm{A^*y^{k+1}+z^{k+1}-c}>\mu \theta^{i_k}\varepsilon_k.
\end{align}}
\end{Step}
\begin{Step}\rm
Set {$x^{k+1}:=x^k-\lambda(A^*y^{k+1}+z^{k+1}-c)$} and $\varepsilon_{k+1}:=\theta^{i_k}\varepsilon_k$. Increase $k$ by $1$ and go to Step 1.
\end{Step}
\end{Algorithm}
\end{longfbox}\vspace*{0.1in}
The augmented Lagrangian $\mathcal{L}_\lambda(y,z,x)$ in \eqref{subproblem numerical} is written now as
\begin{align*}
\mathcal{L}_\lambda(y,z,x)&=\dfrac{1}{2}\norm{y}^2+p^*(z)-\dotproduct{x,A^*y+z-c}+\dfrac{\lambda}{2}\norm{A^*y+z-c}^2.
\end{align*}
Given $i_k\in\N$, we employ the technique in \cite{lsk} to find $(y^{k+1},z^{k+1})$ that satisfy the first inequality in \eqref{subproblem numerical}. Define the mapping $\Psi_k:\R^m\rightarrow\R^n$ and the function $\psi_k:\R^m\rightarrow\R$ by
\begin{align}\label{psik Psik}
\Psi_k(y):=\mathop{\rm argmin}_{z}\mathcal{L}_{\lambda}(y,z,x^k)\;\text{ and }\;\psi_k(y):=\inf_z \mathcal{L}_\lambda(y,z,x^k).
\end{align}
It follows from \cite[Section~3.2]{lsk} that
\begin{align*}
\Psi_k(y)&=\prox_{p^*/\lambda}\brac{\dfrac{x^k}{\lambda}-A^*y+c},\\
\psi_k(y)&=  \dfrac{1}{2}\norm{y}^2+\dfrac{1}{2\lambda}\norm{\prox_{\lambda p}(x^k-\lambda (A^*y-c)) }^2-\dfrac{\norm{x^k}^2}{2\lambda}.
\end{align*}
In addition, $\psi_k$ is strongly convex with constant $1$ and $\mathcal{C}^1$-smooth with the Lipschitzian gradient
\begin{align*}
\nabla \psi_k(y)=y+A\prox_{\lambda p}(x^k-\lambda(A^*y+c))\;\text{ for all }\;y\in\R^m.
\end{align*}

Note that there are explicit formulas for calculating $\prox_{\lambda p}$ and $\prox_{p^*/\lambda}$; see, e.g., \cite[Section~6.9]{beckbook} and \cite{proximity}. However, we do not need these formulas and exact computations of proximal mappings in our algorithms and numerical experiments, since we only focus on solving \eqref{subproblem} {\em inexactly} to illustrate the efficiency of GIALM in what follows.\vspace*{0.05in}

To proceed further, deduce from the definitions of $\Psi_k$ and $\psi_k$ in \eqref{psik Psik} the relationships
\begin{align*}
\psi_k(y)=\mathcal{L}_\lambda(y,\Psi_k(y),x^k)\;\text{ and  }\;\inf_{y,z} \mathcal{L}_\lambda(y,z,x^k)=\inf_y\psi_k(y).
\end{align*}
As a consequence, $(y^{k+1},z^{k+1})$ satisfying the first inequality in \eqref{subproblem numerical} can be obtained from
\begin{align}\label{inf psik}
\psi_k(y^{k+1})-\inf_y \psi_k(y) \le \dfrac{\lambda\theta^{2i_k}\varepsilon_k^2}{2}\;\text{ and }\;z^{k+1}=\Psi_k(y^{k+1}).
\end{align}
The standard characterization of strong convexity gives us the error bound
\begin{align*}
\psi_k(y^{k+1})-\inf_y \psi_k(y)\le \frac{1}{2}\norm{\nabla \psi_k(y^{k+1})}^2,
\end{align*}
which implies that the first inequality in \eqref{inf psik} is satisfied if
\begin{align}\label{stop grad}
\norm{\nabla \psi_k(y^{k+1})}\le \omega_k:= \sqrt{\lambda}\theta^{i_k}\varepsilon_k.
\end{align}
Since $\nabla \psi_k$ is Lipschitz continuous, we can apply the gradient descent method to $\psi_k$ with the stepsize $1/L$, where $L$ is the Lipschitz constant of $\nabla \psi_k$, and find an approximate minimizer $y^{k+1}$ under \eqref{stop grad}. We also put $\varepsilon_1=1$ and $\theta=0.8$ in the initial setting of Algorithm~\ref{IAL}. The two selections of scaling factor are $\mu=3$ and $\mu=1.1,$ which correspond to the versions GIALM-3 and GIALM-1.1, respectively.

\medskip Now we recall the iterative procedure of classical IALM from \cite{rockafellar alm} to solve \eqref{dual problem alm}. Given an initial point $x^1\in\R^n$, IALM uses the following updates for $(y^{k+1},z^{k+1})$ and then $x^{k+1}$ at each iteration:
{\begin{align}\label{al-ine}
 \mathcal{L}_{\lambda}(y^{k+1},z^{k+1},x^k)-\inf_{y,z} \mathcal{L}_{\lambda}(y,z,x^k)\le \delta_k^2\;\text{ and }\;x^{k+1}=x^k-\lambda(A^*y^{k+1}+z^{k+1}-c),
\end{align}}
where the sequence $\set{\delta_k}$ is positive and summable. Using arguments similar to the above tells us that for each $k\in\N$ the pair $(y^{k+1},z^{k+1})$ satisfying the inequality in \eqref{al-ine} can be obtained from 
\begin{align*}
\norm{\nabla \psi_k(y^{k+1})}\le \omega_k:=\sqrt{2}\delta_k\;\text{ and }\;z^{k+1}=\Psi_k(y^{k+1}),
\end{align*}
where $\psi_k$ and $\Psi_k$ are given in \eqref{psik Psik}. In our implementation, we choose $\omega_k:=k^{-q}$ for all $k\in\N$ with some $q>1$ to ensure the summability of $\set{\delta_k}$. Consider also the two specific choices of $q=2$ and $q=3$, which correspond to the versions IALM-2 and IALM-1.5, respectively.\vspace*{0.03in}

Our two numerical experiments presented below are conducted by using a computer with the parameters: 10th Gen Intel(R), Core(TM), i5-10400 (6-Core 12M Cache, 2.9GHz to 4.3GHz), 16GB RAM memory, and the code written in MATLAB R2022b. 
\subsection{An Example in Image Processing}

The setup for the experiment in this subsection follows from that in \cite{beck09}. More specifically, we first assume the reflexive boundary conditions \cite{hansen} for the $256\times 256$ cameraman test image. Then we let this image go through a Gaussian blur of size $9\times 9$ and standard deviation $4$ followed by a standard Gaussian noise with standard deviation $10^{-3}$.  Then vector $b$ in \eqref{primal} represents the observed image, and $A$ is the blurred operator. The regularization parameter is chosen as $\gamma=10^{-4}$. The proximal parameter for the tested methods is chosen as $\lambda=5$, and the initial point is $x^1=b$. In this numerical experiment, we first generate an (approximate) optimal value $f^*$ with its associated (approximate) solution by running GIALM-1.1 in 1000 iterations. The original, blurred, and optimal images are presented in Figure~\ref{fig:optimal}.
\begin{figure}[H]
\centering
\includegraphics[width=.3\textwidth]{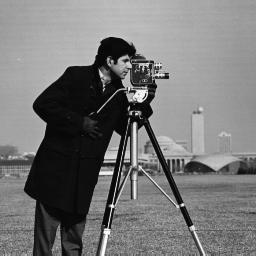}\quad
\includegraphics[width=.3\textwidth]{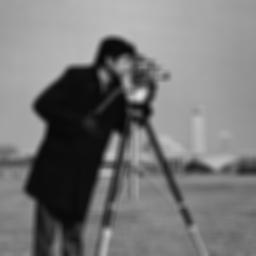}\quad
\includegraphics[width=.3\textwidth]{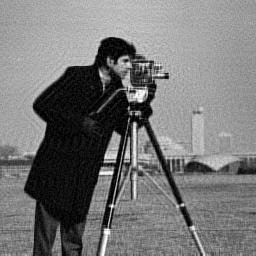} 
\caption{Deblurring of the cameraman}
\label{fig:optimal}
\end{figure}
Then we record the results obtained by the tested methods after running 500 iterations. Their function values compared with $f^*$ and the total number of iterations in the subproblems they execute are given in Figure~\ref{fig:cameraman result} below. It can be seen that GIALM methods have a better performance than the classical IALM methods. Indeed, the left graph shows that after around 5s, the values generated by the GIALM methods are always lower than those for the IALM methods. This is explained in more detail in the right graph, where we see that the IALM methods execute a much larger total number of gradient descent iterations in comparison with the subproblems of the GIALM methods.
\begin{figure}[H]
\centering
\includegraphics[width=.45\textwidth]{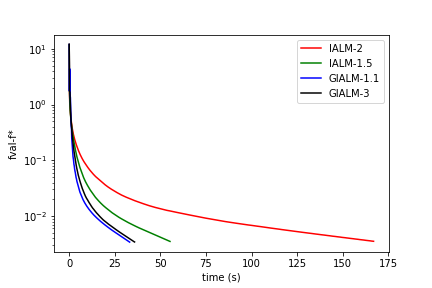}\quad
\includegraphics[width=.45\textwidth]{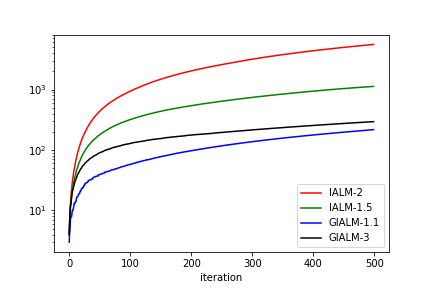}\quad
 \caption{Function values (left) and total number of iterations in subproblems (right)}\label{fig:cameraman result}
\end{figure}\vspace*{-0.2in}

\subsection{Randomly Generated Examples}

In the second experiment, we test the GIALM and IALM methods with different sizes of data. To proceed, the $m\times n$ matrix $A$ and vector $b\in\R^m$ are generated randomly with i.i.d. (identically and independently distributed) standard Gaussian entries. We employ the following stopping criterion used in \cite{kmptjogo,kmptmp,lsk} (see the discussions therein):
\begin{align}\label{stopping criterion dfs}
\eta_k:=\dfrac{\norm{x^k-\prox_{\gamma \norm{\cdot}_1}(x^k-A^*(Ax^k-b))}}{1+\norm{x^k}+\norm{Ax^k-b}}\le 10^{-6}.
\end{align}
We also stop the algorithms if they reach either the time limit of 4000 seconds, or the maximum number of iterations of 200,000. The initial points are chosen as $x^1=0\in\R^n$ for all the algorithms. The selections of the tested parameter $\gamma$ are $10^{-3}$ and $10^{-3}\norm{A^*b}_\infty$, where $\norm{x}_\infty:=\max\set{|x_i|,i=1,\ldots,n}$ for any $x=(x_1,\ldots,x_n)\in\R^n$. While running the tests, we choose the proximal parameter as $\lambda=0.01$ to get the best performance for the tested methods. The detailed information and the results are presented in the table below, where `TN', `iter', `$\eta_k$', and `time' stand, respectively, for test number, the total number of iterations, the residual \eqref{stopping criterion dfs} in the last iteration, and CPU running time of the methods. The tests using the regularization parameter $\gamma=10^{-3}\norm{A^*b}_\infty$ are signified by the asterisks (*) after their test numbers.
\begin{table}[H]
\small
\centering
\resizebox{17cm}{!} 
{\begin{tabular}{|ccc|lll|lll|lll|lll|} 
\hline
\multirow{2}{*}{TN} & \multirow{2}{*}{m} & \multirow{2}{*}{n} & \multicolumn{3}{c|}{IALM-1.5}& \multicolumn{3}{c|}{IALM-2} & \multicolumn{3}{c|}{GIALM-1.1}             & \multicolumn{3}{c|}{GIALM-3} \\ 
\cline{4-15}
&  &  & \multicolumn{1}{c}{iter} & \multicolumn{1}{c}{$\eta_k$} & \multicolumn{1}{c|}{time} & \multicolumn{1}{c}{iter} & \multicolumn{1}{c}{$\eta_k$} & \multicolumn{1}{c|}{time} & \multicolumn{1}{c}{iter} & \multicolumn{1}{c}{$\eta_k$} & \multicolumn{1}{c|}{time} & \multicolumn{1}{c}{iter} & \multicolumn{1}{c}{$\eta_k$} & \multicolumn{1}{c|}{time}  \\ 
\hline
1*& 500 & 1000& 40688  & 1.0E-06  & 68.19   & 40725  & 1.0E-06  & 524.31  & 40495  & 1.0E-06  & 20.78   & 40634  & 1.0E-06  & 25.08    \\
2*& 1000& 1000& 1963   & 1.0E-06  & 28.16   & 2110   & 1.0E-06  & 80.55   & 1985   & 1.0E-06  & 6.21    & 2065   & 1.0E-06  & 10.23    \\
3*& 1000& 2000& 28444  & 1.0E-06  & 879.53  & 8020   & 9.1E-05  & 4000    & 28613  & 1.0E-06  & 291.99  & 28546  & 1.0E-06  & 407      \\
4*& 2000& 2000& 2447   & 1.0E-06  & 979.43  & 2588   & 1.0E-06  & 3999.08 & 2520   & 1.0E-06  & 215.27  & 2560   & 1.0E-06  & 430.91   \\
5*& 2000& 4000& 3917   & 2.0E-04  & 4000    & 608    & 4.8E-03  & 4000    & 28929  & 1.0E-06  & 2594.49 & 21416  & 2.2E-06  & 4000     \\
6*& 4000& 4000& 952    & 9.8E-07  & 3701.33 & 153    & 7.8E-03  & 4000    & 1076   & 9.9E-07  & 1318.44 & 1114   & 1.0E-06  & 2704.56  \\
7 & 500 & 1000& 200000 & 2.1E-05  & 253.4   & 200000 & 2.1E-05  & 1802.58 & 200000 & 3.1E-05  & 96.91   & 200000 & 2.1E-05  & 98.29    \\
8 & 1000& 1000& 75875  & 1.5E-04  & 4000    & 35188  & 5.2E-04  & 4000    & 200000 & 1.0E-05  & 233.99  & 200000 & 1.0E-05  & 244.08   \\
9 & 1000& 2000& 87898  & 8.1E-05  & 4000    & 16867  & 7.2E-04  & 4000    & 200000 & 3.1E-05  & 1548.4  & 200000 & 2.1E-05  & 1626.42  \\
10& 2000& 2000& 1668   & 1.6E-02  & 4000    & 928    & 3.1E-02  & 4000    & 183214 & 6.1E-06  & 4000    & 109370 & 4.7E-05  & 4000     \\
11& 2000& 4000& 6852   & 1.7E-03  & 4000    & 2298   & 4.5E-03  & 4000    & 76571  & 3.6E-04  & 4000    & 58342  & 1.1E-04  & 4000     \\
12& 4000& 4000& 184    & 1.2E-01  & 4000    & 118    & 1.8E-01  & 4000    & 19637  & 1.4E-03  & 4000    & 6219   & 3.8E-03  & 4000     \\
\hline
\end{tabular}}
\caption{Results of IALM and GIALM in random tests}
\end{table}

It can be seen that in Tests~1-6, GIALM-1.1 exhibits the best performance since it achieves the main stopping criterion and has the lowest running time. In Test 7, GIALM-3 is the best since it achieves the smallest residual with the running time lower than IALM-1.5 and IALM-2. In the other tests, GIALM methods also have better performances, i.e., they achieve a smaller residual with the running time less than for the IALM methods. The role of the controlling error of GIALM is presented clearly in Test~12, where GIALM-1.1 executes nearly 20,000 iterations while IALM-2 executes only 118 iterations and thus stagnates at a solution with a much larger residual $\eta_k$. This is because the error of the subproblems in IALM, not being controlled as in GIALM, becomes too small. Therefore, the subproblems of IALM waste much more time to stop. This observation is {confirmed} by the Figure~\ref{fig:test 12} below, which illustrates the residual $\eta_k$ and the error $\omega_k$ in the subproblems of the algorithms.

\begin{figure}[H]
\centering
\includegraphics[width=.45\textwidth]{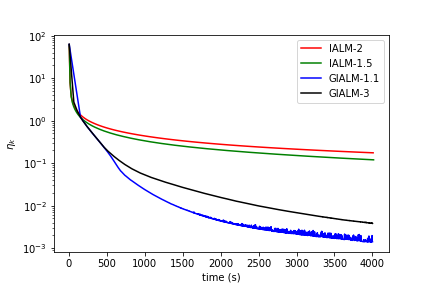}\quad
\includegraphics[width=.45\textwidth]{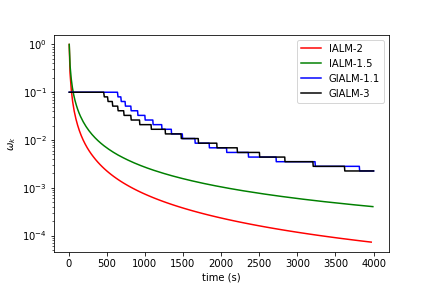}\quad
 \caption{Value of residual $\eta_k$ (left) and error $\omega_k$ (right) in Test 12}\label{fig:test 12}
\end{figure}

In conclusion, GIALM performs better than the classical IALM in these numerical experiments.

\subsection*{Acknowledgements} The authors are very grateful to the editors and the anonymous referees for their helpful remarks and suggestions, which allowed us to significantly improve the original presentation.\vspace*{-0.2in}

\section*{}


\begin{thebibliography}{99}

\bibitem{absil05} P.-A. Absil, R. Mahony and B. Andrews, Convergence of the iterates of descent methods for analytic cost functions, {\em SIAM J. Optim.} {\bf 16} (2005), 531--547.

\bibitem{masoud} M. Ahookhosh, K Amini and S. Bahrami, Two derivative-free projection approaches for systems of large-scale nonlinear monotone equations,
{\em Numer. Algor.} {\bf 64} (2013), 21--42.

\bibitem{attouch10} H. Attouch, J. Bolte, P. Redont and A. Soubeyran,  Proximal alternating minimization and projection methods for nonconvex problems. An approach based on the Kurdyka-\L ojasiewicz property, {\em Math. Oper. Res.} {\bf 35} (2010), 438--457.

\bibitem{attouch13} H. Attouch, J. Bolte and B. F. Svaiter, Convergence of descent methods for semialgebraic and tame problems: Proximal algorithms, forward-backward splitting, and regularized Gauss-Seidel methods, {\em Math. Program.} {\bf 137} (2013), 91--129.

\bibitem{audet17} C. Audet and W. Hare, {\em Derivative-Free and Blackbox Optimization}, Springer, Cham, Switzerland, 2017.

\bibitem{bauschkebook} H. H. Bauschke and P. L. Combettes, {\em Convex Analysis and Monotone Operator Theory in Hilbert Spaces}, 2nd edition, Springer, New York, 2017.

\bibitem{beckbook} A. Beck, {\em First-Order Methods in Optimization}, SIAM, Philadelphia, PA, 2017.

\bibitem{beck09} A. Beck and M. Teboulle, A fast iterative shrinkage-thresholding algorithm for linear inverse problems, {\em SIAM J. Imaging Sci.} {\bf 2} (2009), 183--202.

\bibitem{berahas22} A. S. Berahas, L. Cao, K. Choromanski and K. Scheinberg, A theoretical and empirical comparison of gradient approximations in derivative-free optimization, {\em Found. Comput. Math.} {\bf 22} (2022), 507--560.

\bibitem{berahas18} A. S. Berahas, L. Cao and K. Scheinberg, Global convergence rate analysis of a generic linesearch algorithm with noise, {\em SIAM J. Optim.} {\bf 31} (2021), 1489--1518.

\bibitem{bertsekas00} D. P. Bertsekas and J. N. Tsitsillis, Gradient convergence in gradient methods with errors, {\em SIAM J. Optim.} {\bf 10} (2000), 627--642.

{\bibitem{bog16} L. Bogolubsky, P. Dvurechensky, A. Gasnikov, G. Gusev, Yu. Nesterov, A. M. Raigorodskii, A. Tikhonov and M. Zhukovskii, Learning supervised pagerank with gradient-based and gradient-free optimization methods, {\em Adv. Neural Inf. Process. Syst.} {\bf 29} (2016), 1--9.}

\bibitem{bottou18} L. Bottou, F. E. Curtis and J. Nocedal, Optimization methods for large-scale machine learning, {\em SIAM Rev.} {\bf 60} (2018), 223--311.

\bibitem{byrd12} R. H. Byrd, G. M. Chin and J. Nocedal, Sample size selection in optimization methods for machine learning, {\em Math. Program.} {\bf 134} (2012), 127--155.

\bibitem{carter91} R. G. Carter, On the global convergence of trust region algorithms using inexact gradient information, {\em SIAM J. Numer. Anal.} {\bf 28} (1991), 251--265.

\bibitem{conn09} A. R. Conn, K. Scheinberg and L. N. Vicente, {\em Introduction to Derivative-Free Optimization}, SIAM, Philadelphia, PA, 2009.

\bibitem{curry44} H. B. Curry, The method of steepest descent for non-linear minimization problems, {\em Q. Appl. Math.} {\bf 2} (1944), 258--261.

\bibitem{devolder14} O. Devolder, F. Glineur and Yu. Nesterov, First-order methods of smooth convex optimization with inexact oracle, {\em Math. Program.} {\bf 146} (2014), 37--75. 

{\bibitem{dvu17} P. Dvurechensky, Gradient method with inexact oracle for composite non-convex optimization, arxiv:1703.09180 (2017).}

\bibitem{jonathan15} J. Eckstein and W. Yao, Understanding the convergence of the alternating direction method of multipliers: Theoretical and computational perspectives, {\em Pacific J. Optim.} {\bf 11} (2015), 619--644.

\bibitem{proximity} G. Chierchia, E. Chouzenoux, P. L. Combettes and J.-C. Pesquet, {\em The Proximity Operator Repository}, http://proximity-operator.net/index.html (2016).

{\bibitem{ghadimi13} S. Ghadimi and G. Lan, Stochastic first- and zeroth-order methods for nonconvex stochastic programming, {\em SIAM J. Optim.} {\bf 23} (2013), 2341--2368.}

{\bibitem{ghadimi14} S. Ghadimi, G. Lan and H. Zhang, Mini-batch stochastic approximation methods for nonconvex stochastic composite optimization, {\em Math. Program.} {\bf 155} (2014), 438–--57.}

\bibitem{gilmore95} P. Gilmore and C. T. Kelley, An implicit filtering algorithm for optimization of functions with many local minima, {\em SIAM J. Optim.} {\bf 5} (1995), 269--285.

\bibitem{andreas} A. Griewank and A. Walther, {\em Evaluating Derivatives: Principles and Techniques of Algorithmic Differentiation}, 2nd edition, SIAM, Philadelphia, PA, 2008. 

\bibitem{hansen} P. C. Hansen, J. G. Nagy and D. P. O’Leary, {\em Deblurring Images: Matrices, Spectra, and Filtering}, SIAM, Philadelphia, PH, 2006.

\bibitem{hestenes69} M. R. Hestenes, Multiplier and gradient methods, {\em J. Optim. Theory Appl.} {\bf 4} (1969), 303--320.

\bibitem{solodovbook} A. F. Izmailov and M. V. Solodov, {\em Newton-Type Methods for Optimization and Variational Problems}, Springer, New York, 2014.

\bibitem{kmptjogo} P. D. Khanh, B. S. Mordukhovich, V. T. Phat and D. B. Tran, Generalized damped Newton algorithms in nonsmooth optimization via second-order subdifferentials, {\em J. Global Optim.} {\bf 86} (2023), 93--122.

\bibitem{kmptmp} P. D. Khanh, B. S. Mordukhovich, V. T. Phat and D. B. Tran,
Globally convergent coderivative-based generalized Newton method in nonsmooth optimization, {\em Math. Program.}, doi:10.1007/s10107-023-01980-2 (2023). 

\bibitem{kmt22.1} P. D. Khanh, B. S. Mordukhovich and D. B. Tran, Inexact reduced gradient methods in nonconvex optimization, {\em J. Optim.  Theory Appl.}, doi: 10.1007/s10957-023-02319-9 (2024). 

\bibitem{kurdyka} K. Kurdyka, On gradients of functions definable in o-minimal structures, {\em Ann. Inst. Fourier} {\bf 48} (1998), 769--783. 

\bibitem{lan16} G. Lan and  R. D. C. Monteiro, Iteration-complexity of first-order augmented Lagrangian methods for convex programming, {\em Math. Program.} {\bf 155} (2016), 511--547

\bibitem{lsk} X. Li, D. Sun and K.-C. Toh, A highly efficient semismooth Newton augmented Lagrangian method for solving Lasso problems, {\em SIAM J. Optim.} {\bf 28}, 433--458.

\bibitem{lojasiewicz65} S. \L ojasiewicz, {\em Ensembles Semi-Analytiques}, Institut des Hautes Etudes Scientifiques, Bures-sur-Yvette, France, 1965.

\bibitem{mor-nam} B. S. Mordukhovich and N. M. Nam, {\em Convex Analysis and Beyond, I: Basic Theory}, Springer, Cham, Switzerland, 2022.

{\bibitem{nel14} V. Nedelcu, I. Necoara and Q. Tran-Dinh, Computational complexity of inexact gradient augmented Lagrangian methods: Application to constrained MPC, {\em SIAM J. Control Optim.} {\bf 52} (2014), 3109--3134.}

\bibitem{nesterov05} Yu. Nesterov, Smooth minimization of non-smooth functions, {\em Math. Program.} {\bf 103} (2005), 127--152.

\bibitem{nesterov15} Yu. Nesterov, Universal gradient methods for convex optimization problems, {\em Math. Program.} {\bf 152} (2015), 381--404.

\bibitem{nesterovbook18} Yu. Nesterov, {\em Lectures on Convex Optimization}, 2nd edition, Springer, Cham, Switzerland, 2018.

\bibitem{nielsen15} M. A. Nielsen, {\em Neural Networks and Deep Learning}. Determination Press, New York, 2015.

\bibitem{nocedalbook} J. Nocedal and S. J. Wright, {\em Numerical Optimization}, New York, 1999.

{\bibitem{pas17} A. Patrascu, I. Necoara, and Q. Tran-Dinh, Adaptive inexact fast augmented Lagrangian methods for constrained convex optimization, {\em Optim. Lett.} {\bf 111} (2017), 609--626.}

\bibitem{polyakbook} B. T. Polyak, {\em Introduction to Optimization}, Optimization Software, New York, 1987.

\bibitem{rockafellar76} R. T. Rockafellar, Monotone operators and the proximal point algorithm, {\em SIAM J. Control Optim.} {\bf 14} (1976), 877--898.

\bibitem{rockafellar alm} R. T. Rockafellar, Augmented Lagrangians and applications of the proximal point algorithm in convex programming, {\em Math. Oper. Res.} {\bf 1} (1996) 97--116.

\bibitem{rockafellarbook} R. T. Rockafellar and R. J-B. Wets, {\em Variational Analysis}, Springer, Berlin, 1998.

\bibitem{themelis20} A. Themelis, B. Hermans and P. Patrinos, A new envelope function for nonsmooth DC optimization, {\em Proc. 59th IEEE Conf. Dec. Control}, pp.\ 4697--4702, 2020. 

\bibitem{themelis17} A. Themelis, L. Stella and P. Patrinos, Forward–backward quasi-Newton methods for nonsmooth optimization problems, {\em Comput. Optim. Appl.} {\bf 67} (2017), 443--487.

\bibitem{tuybook} H. Tuy, {\em Convex Analysis and Global Optimization}, 2nd edition, Springer, New York, 2016.

{\bibitem{vasin23} A. Vasin, A. Gasnikov, P. Dvurechenskiy and V. Spokoiny, Accelerated gradient methods
with absolute and relative noise in the gradient, {\em Optim. Methods Softw.} {\bf 38} (2023), 1180–--1229.}
\end{thebibliography}
\end{document}